\documentclass[11pt]{article} 
\usepackage{a4,amssymb, amsmath, amsthm}
\usepackage{graphics,color}
\usepackage{epsfig}
\usepackage{subfigure}
  
\newtheorem{theorem}{Theorem}[section]
\newtheorem{corollary}[theorem]{Corollary}
\newtheorem{lemma}[theorem]{Lemma}
\newtheorem{proposition}[theorem]{Proposition}

\theoremstyle{definition}
\newtheorem{definition}[theorem]{Definition}

\theoremstyle{rem}
\newtheorem{rem}[theorem]{Remark}

\newcommand{\R}{\mathbb{R}}

\begin{document}

\providecommand{\keywords}[1]{{\noindent \textit{Key words:}} #1}
\providecommand{\msc}[1]{{\noindent \textit{Mathematics Subject Classification:}} #1}

\title{A residual a posteriori error estimate for the time--domain boundary element method\\\vskip 0.8cm}
\author{ Heiko Gimperlein\thanks{Maxwell Institute for Mathematical Sciences and Department of Mathematics, Heriot--Watt University, Edinburgh, EH14 4AS, United Kingdom, email: h.gimperlein@hw.ac.uk.}  \thanks{Institute for Mathematics, University of Paderborn, Warburger Str.~100, 33098 Paderborn, Germany.} \and Ceyhun \"{O}zdemir\thanks{Institute of Applied Mathematics, Leibniz University Hannover, 30167 Hannover, Germany.} \thanks{Institute for Mechanics,  Graz University of Technology, 8010 Graz, Austria. \newline H.~G.~acknowledges partial support by ERC Advanced Grant HARG 268105 and the EPSRC Impact Acceleration Account. C.~\"{O}.~was supported by a scholarship of the Avicenna Foundation.}\and  David Stark${}^\ast$ \and Ernst P.~Stephan${}^\ddagger$}\date{}

\maketitle \vskip 0.5cm
\begin{abstract}
\noindent This article investigates residual a posteriori error estimates and adaptive mesh refinements for time-dependent boundary element methods for the wave equation. We obtain reliable estimates for Dirichlet and acoustic boundary conditions which hold for a large class of discretizations. Efficiency of the error estimate is shown for a natural discretization of low order. Numerical examples confirm the theoretical results. The resulting adaptive mesh refinement procedures in $3d$ recover the adaptive convergence rates known for elliptic problems.\\
\end{abstract}

\msc{65N38 (primary); 65M15; 35L67 (secondary)}\\


\section{Introduction}\label{intro}


The efficient numerical treatment of boundary integral equations using adaptive mesh refinement procedures has been extensively investigated for the numerical
solution of homogeneous elliptic problems in unbounded domains \cite{cc, cs}. See \cite{gwinsteph} for a recent exposition. 

In this article we investigate the extension of the a posteriori error analysis and adaptive mesh refinement procedures to initial-boundary value problems for the wave equation, formulated as boundary integral equations in the time-domain \cite{costabel04, sa}. We prove a reliable a posteriori error estimate of residual type for a large class of conforming discretizations.  It is efficient for a time-domain boundary element method on a globally quasi-uniform mesh. The error estimate defines an adaptive mesh refinement procedure, which  recovers the convergence rates known for time-independent screen problems.  

There has been recent interest in the solution of such problems on adapted meshes. Similar to the elliptic case, singularities of the solution may appear at singular points of the boundary, as discussed in \cite{kokotov, kokotov3,plamenevskii}, and in trapping regions. \textcolor{black}{For finite element methods, M\"{u}ller and Schwab used the analytical results to recover quasi-optimal convergence rates on time-independent graded meshes in $2d$ polygons. For boundary element methods in $3d$} time-independent graded meshes have been shown to recover quasi-optimal convergence rates for edge and corner singularities \cite{graded}. First steps towards time-adaptivity for singular temporal behavior are due to Sauter and Veit \cite{sv} in \textcolor{black}{$3$} dimensions, and also convolution quadrature methods with graded, non-adaptively chosen time steps have been studied, for example in \cite{ss}. Gl\"{a}fke \cite{Glaefke} showed first results towards space-time refinements in $2$ dimensions, and in unpublished work Abboud uses ZZ error indicators for computations with space-adaptive mesh refinements for screen problems. 

 \textcolor{black}{The above works have shown the relevance of time-independent adapted meshes not only in simple convex domains, but in realistic complex, heterogeneous geometries. For the noise emission of car tires \cite{comput, Gimperlein} the sound amplification in the cuspidal, non-convex horn geometry between tire and road crucially determines the emitted sound. Time-independent meshes graded into the horn are needed for the accurate computation of the sound emission characteristics, due to the complexity of the geometry, even for a constant-coefficient PDE \cite{graded}. Adaptive meshes are expected to be of use in more general heterogeneous geometries, where the time-averaged indicators resolve persistent spatial inhomogeneities of the solution such as in the current article. \textcolor{black}{The error estimates presented in this article apply to meshes locally refined in both space and time. In 2d, \cite{Glaefke} uses such refinements to resolve space-time singularities, such as travelling wave fronts, with space-time adaptive mesh refinements. The approach requires to assemble $n+1$ matrices in the $n$-th time step, and a feasible 3d implementation will require major algorithmic considerations in future work. In this article we focus on adaptive procedures for heterogeneous geometries, as a key step towards general space-time singularities in 3d.  It complements the work on adaptive time-stepping for a fixed spatial mesh by other authors \cite{sv}.} }\\

To describe the main results, we consider the wave equation
$$\partial_t^2 u - \Delta u = 0\ , \quad u=0\ \text{ for }\ t\leq 0\ ,$$
in the complement \textcolor{black}{$\mathbb{R}^d \setminus \overline{\Omega}$} of a polyhedral domain or screen, with an emphasis on the challenging case $d=3$.
On the boundary $\Gamma = \partial \Omega$ both Dirichlet and acoustic boundary conditions,
$$u = f\ ,\qquad \text{respectively}\qquad \partial_\nu u -\alpha \partial_t u = f \ ,$$
are considered. Here $f$ is given, $\nu$ is the outer unit normal vector to $\Gamma$, and $0<\alpha, \alpha^{-1}\in L^\infty(\Gamma)$.

Following Bamberger and Ha Duong \cite{bh}, we recast the boundary problem as a time dependent boundary integral equation. The Dirichlet problem is equivalent to a hyperbolic variant of Symm's integral equation:
\begin{align} \label{symmhyp}
\mathcal{V} \phi(t,x)= \int_{\mathbb{R}^+} \int_{\Gamma} G(t- \tau,x,y) \phi(\tau,y)\ ds_y \ d\tau = f(t,x)\ .
\end{align}
{Here $\phi$ is sought in a space-time anisotropic Sobolev space $H^{1}_\sigma(\mathbb{R}^+,\widetilde{H}^{-\frac{1}{2}}(\Gamma))$, and} $G$ is a fundamental solution of the wave equation,
\begin{align}\label{eq:green}
 G(t-s,x,y)&= \frac{H(t-s-|x-y|)}{2\pi \sqrt{(t-s)^2+|x-y|^2}}\qquad &\text{in 2d},\\
 G(t-s,x,y)&=\frac{\delta(t-s-|x-y|)}{4\pi |x-y|} &\text{in 3d}.
\end{align}
Here $H$ is the Heaviside function and $\delta$ the Dirac distribution. Our results apply, in particular, to a Galerkin discretization of the weak form of \eqref{symmhyp} in a subspace $V \subset H^{1}_\sigma(\mathbb{R}^+,\widetilde{H}^{-\frac{1}{2}}(\Gamma))$,
\begin{equation}\label{DP}
\int_{\mathbb{R}^+}\int_\Gamma \mathcal{V} \partial_t {\phi}(t,x)\ \psi(t,x)\ ds_x \, d_\sigma t = \int_{\mathbb{R}^+}\int_\Gamma  \partial_t {f}(t,x)\ \psi(t,x)\ ds_x \, d_\sigma t\ ,
\end{equation}
$d_\sigma t = e^{-2\sigma t} dt$ with $\sigma>0$. {For computations we consider subspaces $V = V^{p,q}_{h, \Delta t}$ of tensor products of piecewise polynomials in space and time, defined in Section \ref{prelim}. But also time discretizations based on smooth functions are of interest \cite{sv}.}\\

This article shows that norms of the  residual give upper and lower bounds for the error. The upper bound \eqref{upbd} holds for arbitrary discretizations, not only the Galerkin method \eqref{DP}, and for general meshes: \\

\noindent \textbf{Theorem A:} \emph{Let $\phi  \in H^{1}_\sigma(\mathbb{R}^+,\widetilde{H}^{-\frac{1}{2}}(\Gamma))$ be the solution to \eqref{DP}, and let $\phi_{h,\Delta t} \in H^{1}_\sigma(\mathbb{R}^+,H^{-\frac{1}{2}}(\Gamma))$ such that $\mathcal{R} = \partial_t f - \mathcal{V} \partial_t \phi_{h,\Delta t}\in H^{0}_\sigma(\mathbb{R}^+,H^1(\Gamma))$. Then
\begin{equation}\label{upbd} \|\phi-\phi_{h,\Delta t}\|_{0,-\frac{1}{2}, \Gamma, \ast}^2 \lesssim {\sum_{i,\Delta} \max\{(\Delta t)_i, h_{\Delta}\}\ \|\mathcal{R}\|_{0,1, [t_i,t_{i+1})\times\Delta}^2}\ .\end{equation}
{Let $\Gamma$  be closed and polyhedral. For a globally quasi-uniform mesh on $\Gamma$, let $V = W^{0}_{h,\Delta t}$ the tensor product of cubic splines in time with piecewise constant functions in space. If $\phi_{h,\Delta t} \in V$ is a Galerkin solution of \eqref{DP} in $V$ and $\phi  \in H^{2}_\sigma(\mathbb{R}^+,H^{-\frac{1}{2}}(\Gamma))$,} then for every $\epsilon>0$ \begin{equation}\label{lowbd}\max\{\Delta t, h\} \|\mathcal{R}\|_{0,1-\epsilon,\Gamma}^2 \lesssim \|\phi-\phi_{h,\Delta t}\|_{2,-\frac{1}{2}, \Gamma}^2\ .\end{equation}}

The upper bound \eqref{upbd} is obtained in Corollary \ref{cccor},  the lower bound \eqref{lowbd} in Theorems \ref{lowerboundthm} and \ref{starthm}. Our numerical results illustrate the a posteriori error estimate of Theorem A for time-domain boundary elements based on \eqref{DP}.\\

Note the loss of time derivatives between the upper and lower bound of the error, in the first Sobolev index. The loss is well-known for error estimates for hyperbolic problems \cite{hd}, \textcolor{black}{but see \cite{sut} for current work on a different inf-sup stable bilinear form.} Our arguments generalize to give reliable a posteriori estimates for the acoustic boundary problem, see Section \ref{acousticsection}.\\

The residual error estimate from Theorem A is used to define adaptive mesh refinements in space, based on the four steps Solve, Estimate, Mark, Refine. Numerical experiments confirm the efficiency and reliability of the estimate in examples. For screen problems, \textcolor{black}{where the geometric singularities pose the greatest numerical challenges,}  they recover the convergence rates known for elliptic problems.\\

\textcolor{black}{Our analysis is in line with recent theoretical progress on time domain boundary element methods. Joly and Rodriguez  \cite{jr} discuss practical Galerkin implementations with weight $\sigma = 0$, as opposed to theoretically justified weight functions. Aimi, Diligenti and collaborators use formulations directly related to the conserved energy of the wave equation on a finite time interval [0,T) \cite{aimi, aimi2, aimi3, aimi4}.  At the expense of a slightly more involved weak formulation, the intrinsic coercivity implies the stability and convergence of these methods, rigorously proven for wave problems in a half-space. A detailed exposition of the mathematical background of time domain integral equations and their discretizations is available in the monograph by Sayas \cite{sa}, including methods based on convolution quadrature. See \cite{costabel04, hd} for more concise introductions.  }\\

\textcolor{black}{This current} work builds on the numerical analysis of adaptive boundary element methods for the Laplace equation, both for Symm's integral equation and the hypersingular equation \cite{cc, cs, cms, cmsp}. Work on different types of error indicators in the time-independent case includes ZZ \cite{cp} and Faermann indicators \cite{f1,f2}. \textcolor{black}{Our numerical examples for screens builds on the work by Becache and Ha Duong for crack problems in the time domain \cite{b1,b2, haduongcrack}.} A comparison of different indicators in the time-domain will be the subject of future research.\\

\noindent \emph{Structure of this article:} Section \ref{prelim} recalls the boundary integral operators associated to the wave equation as well as their mapping properties between suitable space-time anisotropic Sobolev spaces.  The Sobolev spaces are discretized using tensor products of piecewise polynomials in space and time. Section \ref{apost} presents a corresponding space-time discretization for the formulation of the Dirichlet problem in terms of the single layer operator and derives a reliable a posteriori error estimate in a simple setting, for globally quasi-uniform meshes, using a canonical approach which readily adapts to other settings. A second subsection analyzes an acoustic boundary problem, a system of equations involving in addition the double layer, adjoint double layer and hypersingular operators. Section \ref{arberror} then  localizes the space-time Sobolev norm to derive the upper estimate for the Dirichlet problem for arbitrary meshes in Theorem A. The upper estimates are complemented by a lower bound  for the error of a Galerkin approximation on globally quasi-uniform meshes in Section \ref{lower}. {The final step of this proof is a lower bound for the best approximation in  Section \ref{sec:2D-Approx}.} 
Section \ref{algo} discusses details of the implementation, \textcolor{black}{and the algorithmic challenges towards efficient space-time adaptive codes in 3d are outlined in Section \ref{towardsst}. Section \ref{experiments} finally presents} numerical experiments which confirm the theoretical results. \textcolor{black}{An appendix shows relevant mapping properties of the boundary integral operators for Sobolev exponents also outside the energy space, Theorem \ref{mapthm2}.}\\

\noindent \emph{Notation:} We write $f \lesssim g$ provided there exists a constant $C$ such that $f \leq Cg$. If the constant
$C$ is allowed to depend on a parameter $\sigma$, we write $f \lesssim_\sigma g$.

\section{Preliminaries and discretization}\label{prelim}

In addition to the single layer operator $\mathcal{V}$, for acoustic boundary problems we require its normal derivative $\mathcal{K}'$, the double layer operator $\mathcal{K}$ and hypersingular operator $\mathcal{W}$ for $x \in \Gamma$, $t>0$: 
\begin{align}
\mathcal{K}\varphi(t,x)&= \int_{\mathbb{R}^+} \int_\Gamma \frac{\partial  \textcolor{black}{G}}{\partial n_y}(t- \tau,x,y)\ \varphi(\tau,y)\ ds_y\ d\tau\ , \label{Koperator}\\
\mathcal{K}' \varphi(t,x)&=  \int_{\mathbb{R}^+}\int_\Gamma \frac{\partial  \textcolor{black}{G}}{\partial n_x}(t- \tau,x,y)\ \varphi(\tau,y)\  ds_y\ d\tau\, , \label{Kpoperator}\\
\mathcal{W} \varphi(t,x)&= -\int_{\mathbb{R}^+} \int_\Gamma \frac{\partial^2  \textcolor{black}{G}}{\partial n_x \partial n_y}(t- \tau,x,y)\ \varphi(\tau,y)\ ds_y\ d\tau \ . \label{Woperator}
\end{align}

Space--time anisotropic Sobolev spaces on the boundary $\Gamma$ provide a convenient setting to study the mapping properties of layer potentials. See \cite{hd, setup} for a detailed exposition. To define them, if $\partial\Gamma\neq \emptyset$, first extend $\Gamma$ to a closed, orientable Lipschitz manifold $\widetilde{\Gamma}$. 

{On $\Gamma$ one defines the usual Sobolev spaces of supported distributions:
$$\widetilde{H}^s(\Gamma) = \{u\in H^s(\widetilde{\Gamma}): \mathrm{supp}\ u \subset {\overline{\Gamma}}\}\ , \quad\ s \in \mathbb{R}\ .$$
Furthermore, ${H}^s(\Gamma)$ is the quotient space $ H^s(\widetilde{\Gamma}) / \widetilde{H}^s({\widetilde{\Gamma}\setminus\overline{\Gamma}})$.} \\
{To write down an explicit family of Sobolev norms, introduce a partition of unity $\alpha_i$ subordinate to a covering of $\widetilde{\Gamma}$ by open sets $B_i$. For diffeomorphisms $\varphi_i$ mapping each $B_i$ into the unit cube $\subset \mathbb{R}^n$, a family of Sobolev norms is induced from $\mathbb{R}^n$, \textcolor{black}{with parameter $\omega \in \mathbb{C}\setminus \{0\}$}:
\begin{equation*}
 ||u||_{s,\omega,{\widetilde{\Gamma}}}=\left( \sum_{i=1}^p \int_{\mathbb{R}^n} (|\omega|^2+|\xi|^2)^s|\mathcal{F}\left\{(\alpha_i u)\circ \varphi_i^{-1}\right\}(\xi)|^2 d\xi \right)^{\frac{1}{2}}\ .
\end{equation*}
The norms for different $\omega \in \mathbb{C}\setminus \{0\}$ are equivalent and $\mathcal{F}$ denotes the Fourier transform. They induce norms on $H^s(\Gamma)$, $||u||_{s,\omega,\Gamma} = \inf_{v \in \widetilde{H}^s(\widetilde{\Gamma}\setminus\overline{\Gamma})} \ ||u+v||_{s,\omega,\widetilde{\Gamma}}$ and on $\widetilde{H}^s(\Gamma)$, $||u||_{s,\omega,\Gamma, \ast } = ||e_+ u||_{s,\omega,\widetilde{\Gamma}}$. {We write $H^s_\omega(\Gamma)$ for $H^s(\Gamma)$, respectively $\widetilde{H}^s_\omega(\Gamma)$ for $\widetilde{H}^s(\Gamma)$, when a norm with a specific $\omega$ is fixed.} $e_+$ extends the distribution $u$ by $0$ from $\Gamma$ to $\widetilde{\Gamma}$. \textcolor{black}{As the norm $||u||_{s,\omega,\Gamma, \ast }$ corresponds to extension by zero, while $||u||_{s,\omega,\Gamma}$ allows  extension by an arbitrary $v$, $||u||_{s,\omega,\Gamma, \ast }$ is stronger than $||u||_{s,\omega,\Gamma}$. Like in the time-independent case the norms are not equivalent whenever $s\in \frac{1}{2} + \mathbb{Z}$ \cite{gwinsteph}.}

{We now define a class of space-time anisotropic Sobolev spaces:
\begin{definition}\label{sobdef}
For {$\sigma>0$ and} $r,s \in\mathbb{R}$ define
\begin{align*}
 H^r_\sigma(\mathbb{R}^+,{H}^s(\Gamma))&=\{ u \in \mathcal{D}^{'}_{+}(H^s(\Gamma)): e^{-\sigma t} u \in \mathcal{S}^{'}_{+}(H^s(\Gamma))  \textrm{ and }   ||u||_{r,s,\sigma,\Gamma} < \infty \}\ , \\
 H^r_\sigma(\mathbb{R}^+,\widetilde{H}^s({\Gamma}))&=\{ u \in \mathcal{D}^{'}_{+}(\widetilde{H}^s({\Gamma})): e^{-\sigma t} u \in \mathcal{S}^{'}_{+}(\widetilde{H}^s({\Gamma}))  \textrm{ and }   ||u||_{r,s,\sigma,\Gamma, \ast} < \infty \}\ .
\end{align*}
$\mathcal{D}^{'}_{+}(E)$ respectively~$\mathcal{S}^{'}_{+}(E)$ denote the spaces of distributions, respectively tempered distributions, on $\mathbb{R}$ with support in $[0,\infty)$, taking values \textcolor{black}{in a Hilbert space $E$. Here we consider $E= {H}^s({\Gamma})$, respectively  $E=\widetilde{H}^s({\Gamma})$.} The relevant norms are given by
\begin{align*}
\|u\|_{r,s,\sigma}:=\|u\|_{r,s,\sigma,\Gamma}&=\left(\int_{-\infty+i\sigma}^{+\infty+i\sigma}|\omega|^{2r}\ \|\hat{u}(\omega)\|^2_{s,\omega,\Gamma}\ d\omega \right)^{\frac{1}{2}}\ ,\\
\|u\|_{r,s,\sigma,\ast} := \|u\|_{r,s,\sigma,\Gamma,\ast}&=\left(\int_{-\infty+i\sigma}^{+\infty+i\sigma}|\omega|^{2r}\ \|\hat{u}(\omega)\|^2_{s,\omega,\Gamma,\ast}\ d\omega \right)^{\frac{1}{2}}\,.
\end{align*}
\end{definition}
\textcolor{black}{They are Hilbert spaces, and we note that the basic case $r=s=0$ is the weighted $L^2$-space with scalar product $\langle u,v \rangle_\sigma := \int_0^\infty e^{-2\sigma t} \int_\Gamma u \overline{v} ds_x\ dt$. Because $\Gamma$ is Lipschitz, like in the case of standard Sobolev spaces} these spaces are independent of the choice of $\alpha_i$ and $\varphi_i$ \textcolor{black}{when $|s|\leq 1$}.\\
}
Using variational arguments, precise mapping properties are well-known for the layer potentials between Sobolev spaces related to the energy. \textcolor{black}{Such estimates have been derived, for example, in \cite{bh, bh2, costabel04, hdthesis, hd}, see \cite{Glaefke} for the precise statement here.}
\begin{theorem}\label{mapthm}
The following operators are continuous for $r \in \mathbb{R}$:
\begin{align*}
\mathcal{V}&: H_\sigma^{r+1}(\mathbb{R}^+, \widetilde{H}^{-\frac{1}{2}}(\Gamma)) \to H_\sigma^{r}(\mathbb{R}^+, H^{\frac{1}{2}}(\Gamma)) \ ,\\ \mathcal{K}' &: H_\sigma^{r+1}(\mathbb{R}^+, \widetilde{H}^{-\frac{1}{2}}(\Gamma)) \to H_\sigma^{r}(\mathbb{R}^+, H^{-\frac{1}{2}}(\Gamma))\ ,\\
\mathcal{K}&: H_\sigma^{r+1}(\mathbb{R}^+, \widetilde{H}^{\frac{1}{2}}(\Gamma)) \to H_\sigma^{r}(\mathbb{R}^+, H^{\frac{1}{2}}(\Gamma)) \ ,\\ \mathcal{W} &: H_\sigma^{r+1}(\mathbb{R}^+, \widetilde{H}^{\frac{1}{2}}(\Gamma)) \to H_\sigma^{r}(\mathbb{R}^+, H^{-\frac{1}{2}}(\Gamma))\ .
\end{align*}
\end{theorem}
\noindent \textcolor{black}{See also \cite{jr} for a detailed discussion of the mapping properties and \cite{sut} for an alternative scale of Sobolev spaces, in both references for Sobolev exponents related to the energy. In the appendix we extend classical arguments for the Laplace equation \cite{cos} to show the following mapping properties also for exponents not related to the energy, as relevant to this article:}
\begin{theorem}\label{mapthm2}
\textcolor{black}{The following operators are continuous for $r \in \mathbb{R}$, $s \in (-\frac12,\frac12 )$:} \begin{align*}
\mathcal{V}&: H_\sigma^{r+1}(\mathbb{R}^+, \widetilde{H}^{-\frac{1}{2}+s}(\Gamma)) \to H_\sigma^{r}(\mathbb{R}^+, H^{\frac{1}{2}+s}(\Gamma)) \ ,\\ \mathcal{K}' &: H_\sigma^{r+1}(\mathbb{R}^+, \widetilde{H}^{-\frac{1}{2}+s}(\Gamma)) \to H_\sigma^{r}(\mathbb{R}^+, H^{-\frac{1}{2}+s}(\Gamma))\ ,\\
\mathcal{K}&: H_\sigma^{r+2}(\mathbb{R}^+, \widetilde{H}^{\frac{1}{2}+s}(\Gamma)) \to H_\sigma^{r}(\mathbb{R}^+, H^{\frac{1}{2}+s}(\Gamma)) \ ,\\ \mathcal{W} &: H_\sigma^{r+2}(\mathbb{R}^+, \widetilde{H}^{\frac{1}{2}+s}(\Gamma)) \to H_\sigma^{r}(\mathbb{R}^+, H^{-\frac{1}{2}+s}(\Gamma))\ .
\end{align*}
\textcolor{black}{If $\Gamma$ is $C^{1,\alpha}$ for some $\alpha >0$, the operators are continuous for  $s \in [-\frac12,\frac12]$.  }
\end{theorem}
\textcolor{black}{For Lipschitz $\Gamma$, the end point estimate $s = \pm\frac12$ is known for elliptic problems, from a deep result by Verchota \cite{verchota}. Its extension to the wave equation is beyond the scope of this article and will be pursued elsewhere. When $\Omega = \mathbb{R}^n_+$, Fourier methods yield improved estimates for $\mathcal{V}$ and $\mathcal{W}$:
\begin{theorem}[\cite{haduongcrack}, pp. 503-506]\label{mapimproved} The following operators are continuous for $r,s \in \mathbb{R}$:
\begin{align*}
& \mathcal{V}:{H}^{r+\frac{1}{2}}_\sigma(\R^+, \widetilde{H}^{s}(\Gamma))\to {H}^{r}_\sigma(\R^+, {H}^{s+1}(\Gamma)) \ , \\
& \mathcal{W}: {H}^{r}_\sigma(\R^+, \widetilde{H}^{s}(\Gamma))\to {H}^{r}_\sigma(\R^+, {H}^{s-1}(\Gamma))\ .
\end{align*}
\end{theorem} 
See also \cite{jr} for a recent discussion of mapping properties. }\\
%
%
%

For simplicity, we assume that the hypersurface $\Gamma$ consists of triangular faces $\Gamma_i$, $\Gamma=\cup_{i=1}^{N_s} \Gamma_i$. Denote by $h_i$ the diameter of $\Gamma_i$, $h=\max_i h_i$. We choose a basis $\{\varphi_j^p\}$ of the space $V_h^p$ of piecewise polynomial functions of degree $p{\geq 0}$ (continuous if $p\geq 1$).  \\

For the time discretization we consider a decomposition of the time interval $\mathbb{R}^+$ into subintervals $I_n=[t_{n-1}, t_n)$ with time step $|I_n|=(\Delta t)_n$, $n=0,1,\dots$. Let $\Delta t = \sup_n (\Delta t)_n$. We denote by $\{\beta^{j,q}\}$ a  corresponding basis  of  the space $V^q_{\Delta t}$ of piecewise  polynomial  functions of degree of $q\textcolor{black}{\geq 0}$ (continuous and vanishing at $t=0$ if $q\geq 1$). In addition to $V^q_{\Delta t}$ we also require the space $W_{\Delta t} \subset V^3_{\Delta t}$ of cubic splines.

The space-time cylinder $\mathbb{R}^+ \times \Gamma$ is discretized by local tensor products  $V^{p,q}_{h,\Delta t}$ in space and time. In the most general case $\mathbb{R}^+ \times \Gamma = \bigcup_i S_i$ is a disjoint union of space-time elements $S_i = [t_{i}^1, t_i^2)\times \Gamma_i$ for some triangles $\Gamma_i$  and time steps $t_{i}^1, t_i^2$. We will call $\bigcup_i S_i$ shape regular if there are constants $\underline{c}, \overline{c}$ such that $\underline{c} h_i^2 \leq |\Gamma_i| \leq \overline{c} h_i^2$ for all $i$. The discrete function space $V^{p,q}_{h,\Delta t}$ consists of functions which restricted to $S_i$ are products of a polynomial of degree $p\geq 0$ in space and a polynomial of degree $q\geq 0$ in time, continuous in space if $p\geq 1$  and continuous and vanishing at $t=0$ if $q\geq 1$. 

We shall particularly focus on the case where the temporal mesh is the same for all triangles, $\mathbb{R}^+ \times \Gamma = \bigcup_{j,k} [t_{j-1},t_j) \times  \Gamma_{k}^j$, so that for all $j$, $\bigcup_k \Gamma_{k}^j$ is a triangulation of $\Gamma$. In this case a basis for $V^{p,q}_{h,\Delta t}$ is given by tensor products $\beta^{j,q}(t) \varphi_{k}^{j,p}(x)$.\\

\textcolor{black}{The space-time meshes generated by the adaptive mesh refinements considered below are always \emph{a refinement} of such a product mesh $\bigcup_{j,k} [t_{j-1},t_j) \times  \Gamma_{k}^j$, but not necessarily themselves a product mesh. Here, we may consider the orthogonal projections $\varPi_{\Delta t}$ from $L^2(\mathbb{R}_+)$ to $V^q_{\Delta t}$, resp.~$\varPi_h$ from $L^2(\Gamma)$ to $V_h^p$. See \cite{setup} for a discussion of their properties and those of their composition $\Pi_{\Delta t}\Pi_{h}$.  {Furthermore, we define $W^{p}_{h,\Delta t}= V_h^p \otimes W_{\Delta t}$. }}\\

\textcolor{black}{Note the following approximation properties for such meshes, see also Proposition 3.54 of \cite{Glaefke}:}
\begin{lemma}\label{interp}
Let $f \in H^{r}_\sigma(\mathbb{R}^+, H^m(\Gamma){\cap \widetilde{H}^{s}(\Gamma)})$, $0<m\leq q+1$, $0<r\leq p+1$, $s\leq r$, $|l|\leq \frac{1}{2}$ such that $ls\geq 0$. Then there exists $f_{h,\Delta t} = \Pi_{\Delta t}\Pi_h f$ such that for all  $ l,s\leq 0$
\begin{align*}\label{eq:approx}
\|f-f_{h,\Delta t} f\|_{s,l,\Gamma} &\leq C (h^\alpha + (\Delta t)^\beta)||f||_{r,m,\Gamma}\ ,
\end{align*}
where $\alpha = \min\{m-l, m-\frac{m(l+s)}{m+r}\}$, $\beta = \min\{m+r-(l+s), m+r-\frac{m+r}{m}l\}$. If $l,s>0$, $\beta = m+r-(l+s)$.
\end{lemma}

\section{A posteriori error estimates -- reliability}\label{apost}

\subsection{Dirichlet problem}

We recall the basic properties of the bilinear form $$B_{D}(\phi, \psi) = \int_{\mathbb{R}^+}\int_\Gamma \mathcal{V} \partial_t {\phi}(t,x)\ \psi(t,x)\ ds_x \, d_\sigma t$$ of the Dirichlet problem. \\

As shown in \cite{hd}, the bilinear form is continuous, and also weakly coercive:
\begin{proposition}\label{DPbounds} For every $\phi,\psi \in H^1_\sigma( \mathbb{R}^+, H^{-\frac{1}{2}}(\Gamma))$ there holds:
$$|B_{D}(\phi,\psi)| \lesssim \|\phi\|_{1,-\frac{1}{2},\Gamma, \ast} \|\psi\|_{1,-\frac{1}{2}, \Gamma,\ast}$$
and
$$\|\phi\|_{0,-\frac{1}{2},\Gamma,\ast}^2 \lesssim B_{D}(\phi,\phi) . $$
\end{proposition}
Note the loss of a time derivative between the upper and lower estimates. Alternative inf-sup stable bilinear forms for the Dirichlet problem are the content of current work \cite{sut}.\\  

We consider a conforming Galerkin discretization of the Dirichlet problem \eqref{DP} in a subspace $V \subset H^1_\sigma(\mathbb{R}^+,\widetilde{H}^{-\frac{1}{2}}(\Gamma))$, which reads as follows: Find $\phi_{h,\Delta t} \in V$ such that
\begin{equation}\label{DPdisc}
B_{D}(\phi_{h,\Delta t},\psi_{h,\Delta t})=\langle \partial_t f,\psi_{h,\Delta t}\rangle\ ,
\end{equation}
for all $\psi_{h, \Delta t} \in V$. \\

The well-posedness of the continuous and discretized problems are a basic consequence of Proposition \ref{DPbounds}:
\begin{corollary}
Let $f \in H^2_\sigma(\mathbb{R}^+,H^{\frac{1}{2}}(\Gamma))$. Then the Dirichlet problem \eqref{DP} and its discretization \eqref{DPdisc} admit unique solutions $\phi \in H^1_\sigma(\mathbb{R}^+,\widetilde{H}^{-\frac{1}{2}}(\Gamma))$, $\phi_{h,\Delta t} \in V$, and the estimates $$\|\phi\|_{1,-\frac{1}{2},\Gamma,\ast}, \|\phi_{h,\Delta t}\|_{1,-\frac{1}{2},\Gamma,\ast} \lesssim \|f\|_{2,\frac{1}{2},\Gamma}$$ hold.
\end{corollary}
We note the Galerkin orthogonality:
\begin{equation*}
B_{D}(\phi-\phi_{h,\Delta t},\psi_{h,\Delta t})=0 \qquad \forall \psi_{h,\Delta t} \in V\,.
\end{equation*}
Using ideas going back to Carstensen \cite{cc} and Carstensen and Stephan \cite{cs} for the boundary element method for elliptic problems, we obtain an a posteriori error estimate for the Galerkin solution to the Dirichlet problem on globally quasi-uniform meshes.

\begin{theorem}\label{DPreliable}
Given $V = V^{p,q}_{h, \Delta t}$ associated to a globally quasi-uniform discretization of $\Gamma$, let $\phi \in H^1_\sigma(\mathbb{R}^+,\widetilde{H}^{-\frac{1}{2}}(\Gamma))$, $\phi_{h,\Delta t} \in V$ the solutions to \eqref{DP}, resp.~\eqref{DPdisc}. Assume that $\mathcal{R} = \partial_t{f} - \mathcal{V} \partial_t{\phi}_{h,\Delta t}\in H^0_\sigma(\mathbb{R}^+,H^{1}(\Gamma))$. Then
\begin{align*}
 \|\phi-\phi_{h,\Delta t}\|_{0,-\frac{1}{2},\Gamma,\ast}^2 &\lesssim \|\textcolor{black}{\mathcal{R}}\|_{0,1,\Gamma}\big(
\Delta t \|\partial_t \mathcal{R}\|_{0,0,\Gamma} + \|h \cdot \nabla \mathcal{R}\|_{0,0, \Gamma} \big)\\
& \lesssim \max\{\Delta t, h\} (\|\partial_t \mathcal{R}\|_{0,0,\Gamma}^2 + \|\nabla \mathcal{R}\|_{0,0,\Gamma}^2)\ .
\end{align*}
\end{theorem}\textcolor{black}{\begin{rem} The estimate generalizes to arbitrary subspaces $V$ in place of $V^{p,q}_{h, \Delta t}$, in particular discretizations with smooth ansatz functions in time are of interest \cite{sv}.\\
a) With the endpoint estimate in Theorem \ref{mapthm2}, the single--layer potential maps $H^{1}_\sigma(\mathbb{R}^+, L^2(\Gamma))$ continuously to $H^{0}_\sigma(\mathbb{R}^+, H^{1}(\Gamma))$, and $\mathcal{V} \dot{\phi}_{h,\Delta t}$ belongs to $H^0_\sigma(\mathbb{R}^+,H^{1}(\Gamma))$ if, for example, $\phi_{h,\Delta t}\in H^{2}_\sigma(\mathbb{R}^+, L^2(\Gamma))$. The a posteriori estimate is therefore valid for discretizations by piecewise constant functions in space and $C^1$--continuous splines in time. In practice, as noted in \cite{jr}, the loss of time derivatives in the mapping properties of Theorems \ref{mapthm} and \ref{mapthm2} is not sharp, and $\mathcal{R} \in H^0_\sigma(\mathbb{R}^+,H^1(\Gamma))$ can also be expected for lower-order discretizations in time.\\
b) In practice, we will here use $\Delta t \|\partial_t \mathcal{R}\|_{0,0,\Gamma} + \|h \cdot \nabla \mathcal{R}\|_{0,0, \Gamma}$ as an error indicator.\end{rem}}
\begin{proof}
We first note that for all $\psi_{h,\Delta t} \in V^{p,q}_{h,\Delta t}$
\begin{align*}
&\|\phi-\phi_{h,\Delta t}\|_{0,-\frac{1}{2},\Gamma, \ast}^2\lesssim B_{D}(\phi-\phi_{h,\Delta t},\phi-\phi_{h,\Delta t}) \\
& = \int_{\mathbb{R}^+}\int_\Gamma \partial_t{f} (\phi-\phi_{h,\Delta t})\ ds_x\ d_\sigma t -  B_{D}(\phi_{h,\Delta t},\phi-\phi_{h,\Delta t}) \\
& =\int_{\mathbb{R}^+} \int_\Gamma \partial_t{f} (\phi-\psi_{h,\Delta t})\ ds_x \ d_\sigma t -  B_{D}(\phi_{h,\Delta t},\phi-\psi_{h,\Delta t})  \\
& = \int_{\mathbb{R}^+}\int_\Gamma (\partial_t{f} - \mathcal{V} \partial_t{\phi}_{h,\Delta t}) (\phi-\psi_{h,\Delta t})\ ds_x \ d_\sigma t\ .
\end{align*}
The last term may be estimated by:
\begin{align*}
& \int_{\mathbb{R}^+} \int_\Gamma (\partial_t{f} - \mathcal{V} \dot{\phi}_{h,\Delta t}) (\phi-\psi_{h,\Delta t})\ ds_x \ d_\sigma t \leq \ \|\mathcal{R}\|_{0, \frac{1}{2},\Gamma} \|\phi-\psi_{h,\Delta t}\|_{0, -\frac{1}{2},\Gamma, \ast}\ .
\end{align*}
We use $\psi_{h,\Delta t} = \phi_{h,\Delta t}$ together with the interpolation inequality
$$\|\mathcal{R}\|^2_{0, \frac{1}{2},\Gamma}\leq \|\mathcal{R}\|_{0,0,\Gamma}\|\mathcal{R}\|_{0,1,\Gamma}\ .$$
As the residual is perpendicular to $V^{p,q}_{h,\Delta t}$,
\begin{align*}
\|\mathcal{R}\|_{0,0,\Gamma}^2 &= \langle \mathcal{R},\mathcal{R}\rangle = \langle \mathcal{R},\mathcal{R}-\widetilde{\psi}_{h,\Delta t}\rangle\leq \|\mathcal{R}\|_{0,0,\Gamma}\|\mathcal{R}-\widetilde{\psi}_{h,\Delta t}\|_{0,0,\Gamma}
\end{align*}
for all $\widetilde{\psi}_{h,\Delta t} \in V^{p,q}_{h,\Delta t}$, we obtain
$$\|\mathcal{R}\|_{0,0,\Gamma} \leq \inf\{\|\mathcal{R}-\widetilde{\psi}_{h,\Delta t}\|_{0,0,\Gamma}: \widetilde{\psi}_{h,\Delta t} \in V^{p,q}_{h,\Delta t}\}\ .$$
Choosing $\widetilde{\psi}_{h,\Delta t} = \Pi_{\Delta t}\Pi_h \mathcal{R}$, based on the interpolation operator defined earlier, we obtain
$$\|\mathcal{R}\|_{0,0,\Gamma}\lesssim \Delta t \|\partial_t \mathcal{R}\|_{0,0,\Gamma} + \|h \cdot \nabla \mathcal{R}\|_{0,0,\Gamma} \ .$$
The theorem follows.
\end{proof}

\subsection{Acoustic boundary problems}\label{acousticsection}

Recall the wave equation with inhomogeneous acoustic boundary conditions
$$\partial_t^2 u - \Delta u = 0\ \text{ on $\mathbb{R}_+\times \mathbb{R}^d \setminus \overline{\Omega}$ },\ \partial_\nu u -\alpha \partial_t u = f\ \text{ on $\mathbb{R}_+\times \Gamma$ }, \quad u=0\ \text{ for }\ t\leq 0\ .$$
For scattering problems $f = - \partial_\nu u_{inc}+\alpha \partial_t u_{inc}$ is determined from an incoming wave $u_{inc}$.  

For a finite or infinite time interval $[0,T]$ we introduce the bilinear form
\begin{equation}\label{eq:bilinear_recall}
 a_T((\varphi,p),(\psi,q))=\int_0^T \int_\Gamma \left( \alpha \dot{\varphi} \dot{\psi}  +  \frac{1}{\alpha} p q +  \mathcal{K}' p \dot{\psi}
- \mathcal{W} \varphi \dot{\psi} + \mathcal{V} \dot{p} q +K \dot{\varphi} q \right) ds_x\, dt \,.
\end{equation}
With
\begin{equation}\label{eq:rechte_recall}
 l(\psi, q)=\int_0^T \int_\Gamma F \dot{\psi} \,ds_x\, dt+\int_0^T \int_\Gamma \frac{G q}{\alpha}\, ds_x\, dt\,,
\end{equation}
where $F = -2\partial_\nu u_{inc}$, $G=-2\alpha \partial_t u_{inc}$, we consider the variational formulation for the wave equation in $\mathbb{R}^d$ with acoustic boundary conditions on $\Gamma$:\\

\noindent Find $(\varphi,p) \in H^1([0,T],\widetilde{H}^{\frac{1}{2}}(\Gamma)) \times H^1([0,T],L^2(\Gamma))$ such that
\begin{equation}\label{eq:acoustic_recall}
a_T((\varphi,p), (\psi,q)) = l(\psi, q)
\end{equation}
for all $(\psi,q)\in H^1([0,T],{H}^{\frac{1}{2}}(\Gamma)) \times H^1([0,T],L^2(\Gamma))$.\\

{Note that $\sigma$ may be set to $0$ in the definition of the Sobolev spaces when $T<\infty$}. The acoustic problem is equivalent to the wave equation with acoustic boundary conditions \cite{hd}. Its discretization reads:\\

\noindent Find $(\varphi_{h,\Delta t}, p_{h,\Delta t}) \in V^{p_1,q_1}_{h, \Delta t} \times V^{p_2,q_2}_{h, \Delta t}$ such that
\begin{equation}\label{eq:acoustic_recallh}
a_T((\varphi_{h,\Delta t},p_{h,\Delta t}), (\psi_{h,\Delta t},q_{h,\Delta t})) = l(\psi_{h,\Delta t}, q_{h,\Delta t})
\end{equation}
for all $(\psi_{h,\Delta t},q_{h,\Delta t})\in V^{p_1,q_1}_{h, \Delta t} \times V^{p_2,q_2}_{h, \Delta t}$.}\\

The following well-posedness holds: 
\begin{proposition}
Let $F \in H^2([0,T],H^{-\frac{1}{2}}(\Gamma))$, $G \in H^1([0,T],H^{0}(\Gamma))$. Then the weak form \eqref{eq:acoustic_recall} of the acoustic problem  and its discretization \eqref{eq:acoustic_recallh} admit unique solutions $(\varphi,p)\in  H^1([0,T],\widetilde{H}^{\frac{1}{2}}(\Gamma))\times H^{1}([0,T], L^2(\Gamma))$, resp.~$(\varphi_{h,\Delta t},p_{h,\Delta t})\in V^{p_1,q_1}_{h, \Delta t} \times V^{p_2,q_2}_{h, \Delta t}$, which depend continuously on the data.
\end{proposition}

We specifically note that the bilinear form $a_T$ satisfies a (weaker) coercivity estimate:
$$
\|p\|^2_{0,0,\Gamma}+\|\dot{\varphi}\|_{0,0,\Gamma}^2 \lesssim a_{T}((\varphi,p),(\varphi,p))\ .
$$
This follows from Equation (64) of \cite{hd}, 
$$ a_{T}((\varphi,p),(\varphi,p)) = 2E(T) + \int_0^T \int_\Gamma \left( \alpha (\partial_t{\varphi}) (\partial_t{\varphi})  + \frac{1}{\alpha} p^2 \right) ds_x\, dt \ ,$$
where $E(T)= \frac{1}{2}\int_{\mathbb{R}^d \setminus \overline{\Omega}} \left\{(\partial_t u)^2 + (\nabla u)^2\right\}  dx$ is the total energy at time $T$.

We state a simple a posteriori estimate.

\begin{theorem}
Let $(\varphi,p) \in H^1([0,T],\widetilde{H}^{\frac{1}{2}}(\Gamma))\times H^{1}([0,T], L^2(\Gamma))$ be the solution to \eqref{eq:acoustic_recall}, $(\varphi_{h,\Delta t},p_{h,\Delta t})\in V^{p_1,q_1}_{h, \Delta t} \times V^{p_2,q_2}_{h, \Delta t}$ the solution to the discretization \eqref{eq:acoustic_recallh}. Assume that
\begin{align*}
R_1 &= F - \alpha \dot{\varphi}_{h,\Delta t}+2\mathcal{K}'p_{h,\Delta t} - 2\mathcal{W}\varphi_{h,\Delta t}\in L^2([0,T],L^2(\Gamma))\ ,\\
R_2 &= G + \alpha^{-1}p_{h,\Delta t}+2\mathcal{V}\dot{p}_{h,\Delta t} - 2\mathcal{K}\dot{\varphi}_{h,\Delta t}\in L^2([0,T],L^2(\Gamma))\ .
\end{align*}
Then 
\begin{align*}
 & \|p-p_{h,\Delta t}\|_{0,0,\Gamma}+\|\dot{\varphi} -\dot{\varphi}_{h,\Delta t}\|_{0,0,\Gamma} \lesssim \|R_1\|_{0,0,\Gamma} + \|R_2\|_{0,0,\Gamma}\ .
\end{align*}
\end{theorem}
\textcolor{black}{\begin{rem}\label{regremark}From the mapping properties in Theorem \ref{mapthm2}, the assumptions are satisfied if $F, G \in L^2([0,T],L^2(\Gamma))$ and $(\varphi_{h,\Delta t},p_{h,\Delta t}) \in  H^2([0,T],\widetilde{H}^{1}(\Gamma))\times H^{1}([0,T], L^2(\Gamma))$. For example, this is true for discretizations with $\varphi_{h,\Delta t}$ piecewise linear  in space, higher-order spline in time, and $p_{h,\Delta t}$ piecewise constant in space, piecewise linear in time. In practice, as noted in \cite{jr} for the single-layer potential, the loss of time derivatives in the mapping properties of Theorems \ref{mapthm} and \ref{mapthm2} is not sharp, and $R_1, R_2 \in L^2([0,T],L^2(\Gamma))$ can also be expected for lower-order discretizations in time.\end{rem}}
\begin{proof}
For every $(\psi_{h,\Delta t}, q_{h,\Delta t}) \in V^{p_1,q_1}_{h, \Delta t} \times V^{p_2,q_2}_{h, \Delta t}$ we have
\begin{align*}
& \|p-p_{h,\Delta t}\|^2_{0,0,\Gamma}+\|\dot{\varphi} -\dot{\varphi}_{h,\Delta t}\|_{0,0,\Gamma}^2  \\
&\lesssim a_{T}((\varphi -\varphi_{h,\Delta t},p-p_{h,\Delta t}),(\varphi -\varphi_{h,\Delta t},p-p_{h,\Delta t}))\\
&=\langle R_1, \dot{\varphi} -\dot{\psi}_{h,\Delta t}\rangle + \langle R_2, p -q_{h,\Delta t}\rangle\\
& \leq  \|R_1\|_{0,0,\Gamma}\|\dot{\varphi} -\dot{\psi}_{h,\Delta t}\|_{0,0,\Gamma}+\|R_2\|_{0,0,\Gamma}\|p -q_{h,\Delta t}\|_{0,0, \Gamma}\\
& \leq  (\|R_1\|_{0,0,\Gamma}+\|R_2\|_{0,0,\Gamma}) (\|p-q_{h,\Delta t}\|_{0,0,\Gamma}+\|\dot{\varphi} -\dot{\psi}_{h,\Delta t}\|_{0,0,\Gamma})\ .
\end{align*}
The assertion is obtained by choosing $(\psi_{h,\Delta t}, q_{h,\Delta t})=(\varphi_{h,\Delta t},p_{h,\Delta t})$.
\end{proof}

Naturally, for a quasi-uniform discretization of $\Gamma$ and under stronger assumptions on $R_1, R_2$ we may obtain powers of $h$ and $\Delta t$ on the right hand side by the following argument:\\
As in the proof of Theorem \ref{DPreliable}, $\langle R_1,\dot{\widetilde{\psi}}_{h,\Delta t}\rangle = \langle R_2, \widetilde{q}_{h,\Delta t}\rangle =0$ for all $(\widetilde{\psi}_{h,\Delta t}, \widetilde{q}_{h,\Delta t}) \in V^{p_1,q_1}_{h, \Delta t} \times V^{p_2,q_2}_{h, \Delta t}$. Hence
\begin{align*}
\|R_2\|_{0,0,\Gamma}^2 &= \langle R_2,R_2\rangle = \langle R_2,R_2-{\widetilde{q}}_{h,\Delta t}\rangle\leq \|R_2\|_{0,0,\Gamma}\|R_2-\widetilde{q}_{h,\Delta t}\|_{0,0,\Gamma}.
\end{align*}
Choosing $\widetilde{q}_{h,\Delta t} =\Pi_{\Delta t} \Pi_{h}R_2$ yields
$$\|R_2\|_{0,0, \Gamma}\lesssim \Delta t \|\partial_t R_2\|_{0,0,\Gamma} + \|h \cdot \nabla_\Gamma R_2\|_{0,0,\Gamma} + \Delta t\|h \cdot \nabla_\Gamma \partial_t R_2\|_{0,0,\Gamma}$$
provided $R_2 \in H^1([0,T], H^1(\Gamma))$.\\
Assuming $R_1 \in H^1([0,T], H^1(\Gamma))$, we similarly have
\begin{align*}
\|R_1\|_{0,0,\Gamma}^2 &= \langle R_1,R_1\rangle = \langle R_1,R_1-\dot{\widetilde{\psi}}_{h,\Delta t}\rangle \leq \|R_1\|_{\frac{1}{2}-s,0,\Gamma}\|\int_0^t R_1-\widetilde{\psi}_{h,\Delta t}\|_{\frac{1}{2}+s,0,\Gamma}\ .
\end{align*}
Choosing $\widetilde{\psi}_{h,\Delta t} = \Pi_{\Delta t} \Pi_{h}\int_0^t R_1$ and $s=\frac{1}{2}$ results as in Lemma \ref{interp} in
\begin{align*}
\|R_1\|_{0,0,\Gamma} &\lesssim \Delta t \|\partial_t R_1\|_{0,0,\Gamma} + \|h\cdot \nabla_\Gamma R_1\|_{0,0,\Gamma} +\Delta t\|h\cdot \nabla_\Gamma \partial_t R_1\|_{0,0,\Gamma}\ .
 \end{align*}
Altogether,
\begin{align*}
 &\|p-p_{h,\Delta t}\|_{0,0,\Gamma}+\|\dot{\varphi} -\dot{\varphi}_{h,\Delta t}\|_{0,0,\Gamma} + \|{\varphi} -{\varphi}_{h,\Delta t}\|_{0,\frac{1}{2},\Gamma}\\
 &\lesssim \|R_1\|_{0,0,\Gamma} + \|R_2\|_{0,0,\Gamma}\\
 & \lesssim \sum_{i=1}^2 \Delta t \|\partial_t R_i\|_{0,0,\Gamma} + \|h \cdot \nabla R_i\|_{0,0,\Gamma}  + \Delta t\|h \cdot \nabla \partial_t R_i\|_{0,0,\Gamma}\ .
\end{align*}

\textcolor{black}{As in Remark \ref{regremark}, a sufficient condition for $R_1, R_2 \in H^1([0,T], H^1(\Gamma))$ is given by $F,G \in H^1([0,T], H^1(\Gamma))$ and $(\varphi_{h,\Delta t},p_{h,\Delta t}) \in  H^3([0,T],H^{2}(\Gamma))\times H^{2}([0,T], H^1(\Gamma))$. This requires an extension of Theorem \ref{mapthm2} to Sobolev indices $s > \frac{1}{2}$, available for example using pseudodifferential operator techniques for smooth $\Gamma$. Again, because Theorems \ref{mapthm} and \ref{mapthm2} are not sharp, less time regularity might be required in practice.}

\section{Error estimates for general discretizations}\label{arberror}

This section generalizes the results for the single layer potential without any assumptions on the underlying meshes.

We recall Lemma 3 in \cite{graded}:
\begin{lemma}\label{cc21}
Let $f_1, \dots, f_n \in H^r_\sigma(\mathbb{R}^+, \widetilde{H}^s(\Gamma))$, $0 \leq s \leq 1$, $r\geq 0$, such that $f_j f_k =0$ for $j \neq k$. Let $\omega_j$ be the interior of the support of $f_j$ with $\overline{\omega_j} = \mathrm{supp}\ f_j$. Then 
$$\Big\|\sum_{j=1}^n f_j \Big\|_{r,s,\Gamma, \ast}^2 \leq C \sum_{j=1}^n \left\|f_j\right\|_{r,s, \omega_j, \ast}^2\ .$$
The constant $C$ depends on $\Gamma$, but does not depend on $f_j$ or on $n$.
\end{lemma}

The lemma will be applied to a finite partition of unity $\Phi$ given by non-negative Lipschitz functions $\{\phi_j\}_{j=1}^M$ on $\Gamma$ such that $\sum_{j=1}^M \phi_j = 1$.
\begin{definition}
The overlap of the partition of unity $\Phi$ is defined as $K(\Phi) = \max_{j} \mathrm{card}\{k : \phi_k \phi_j \neq 0\}$.
\end{definition}

For a partition of unity associated to a triangulation, $M$ tends to infinity as the mesh size decreases, while the overlap may be much smaller. We note a crucial observation from \cite{cms}, Lemma 3.1:
\begin{lemma}\label{cc31}
Let $\Phi$ be a finite partition of unity of $\Gamma$ with overlap $K(\Phi)$. Then there exists a partition of $\{1,\dots,M\}$ into $K \leq K(\Phi)$ non-empty subsets $M_1, \dots, M_K$, such that $\bigcup_{j=1}^K M_j = \{1, \dots, M\}$, $M_j \cap M_k = \emptyset$ if $j \neq k$ and for all $l \in \{1,\dots, K\}$ and $j,k \in M_l$ with $j \neq k$, $\phi_j \phi_k = 0$ on $[0,T]\times \Gamma$.
\end{lemma}

Consider  a space-time discretization $\mathbb{R}^+ \times \Gamma = \bigcup_{l,j} [t_{l-1},t_l) \times  \Gamma_{j}^l$ subordinate to product mesh, and consider the partition of unity given by the associated hat functions $\{\phi_j^l\}_j$ on $\Gamma$. On $\mathbb{R}^+$ we consider a partition of unity $\{\psi_l\}_{l=0}^\infty$ such that $\psi_l$ is supported in the interval $((l-\frac{1}{2})\Delta t, (l+\frac{3}{2})\Delta t)$, and therefore $\psi_l \psi_{l'}=0$ whenever $l \in
\widetilde{M}_0 =2\mathbb{N}$, $l' \in \widetilde{M}_1 = 2\mathbb{N}+1$. \\
We obtain a  partition of unity $\{\widetilde{\psi}_{l,j}=\psi_l \otimes \phi_j^l\}_{l,j}$ on $\mathbb{R}^+\times \Gamma$.
\begin{theorem}\label{cct31}
Let $\Gamma' \subset \Gamma$ be connected. and let $\Phi$ be a finite partition of unity with overlap $K(\Phi)$. Then for any $f \in H^r_\sigma(\mathbb{R}^+, \widetilde{H}^s(\Gamma'))$ and any $0\leq s\leq1$, we have 
$$\|f\|_{r,s,\Gamma', \ast}^2 \lesssim_\sigma K(\Phi) \sum_{l=0}^\infty\sum_{j=1}^M \|f \widetilde{\psi}_{l,j} \|_{r,0,\Gamma'}^{2(1-s)}\|f \widetilde{\psi}_{l,j}\|_{r,1,\Gamma'}^{2s}\ .$$
\end{theorem}
\begin{proof}
We show \begin{equation}\label{localizationineq}\|f\|_{r,s,\Gamma', \ast}^2 \lesssim 2K(\Phi) \sum_{l=0}^\infty \sum_{j=1}^M \|f \widetilde{\psi}_{l,j}\|_{r,s,\Gamma',\ast}^{2}\ .\end{equation}
The assertion then follows from the interpolation estimate
$$\|f \widetilde{\psi}_{l,j}\|_{r,s,\Gamma', \ast} \lesssim \|f \widetilde{\psi}_{l,j}\|_{r,0,\Gamma'}^{1-s}\|f \widetilde{\psi}_{l,j}\|_{r,1,\Gamma'}^{s}\ .$$
To show \eqref{localizationineq}, we consider a partition $M_1, \dots, M_K$ as in Lemma \ref{cc31}. Then $f = \sum_{l=0}^\infty\sum_{j=1}^M \widetilde{\psi}_{l,j} f$, so that
$$ \|f\|_{r,s,\Gamma', \ast}^2 =\|\sum_{m=0}^1 \sum_{l \in \widetilde{M}_m}\sum_{k=1}^K\sum_{j\in M_k}\widetilde{\psi}_{l,j} f\|_{r,s,\Gamma', \ast}^2 \leq 2K \sum_{m=0}^1\sum_{k=1}^K \|\sum_{l \in \widetilde{M}_m}\sum_{j\in M_k}\widetilde{\psi}_{l,j} f\|_{r,s,\Gamma', \ast}^2 \ .$$
With Lemma \ref{cc21} and $I_l=((l-1)\Delta t, (l+2) \Delta t)$, $$\|\sum_{l \in \widetilde{M}_m}\sum_{j\in M_k}\widetilde{\psi}_{l,j} f\|_{r,s,\Gamma', \ast}^2 \leq \sum_{j\in M_k} \|\sum_{l \in \widetilde{M}_m}\widetilde{\psi}_{l,j} f\|_{r,s, \omega_j,\ast}^2 \leq \sum_{l \in \widetilde{M}_m}\sum_{j\in M_k} \|\widetilde{\psi}_{l,j} f\|_{r,s,I_l \times \omega_j,\ast}^2\ ,$$ so that 
$$\|f\|_{r,s,\Gamma', \ast}^2 \leq 2 K \sum_{m=0}^1 \sum_{l \in \widetilde{M}_m}\sum_{k=1}^K \sum_{j\in M_k} \|\widetilde{\psi}_{l,j} f\|_{r,s,I_l\times \omega_j,\ast}^2 = 2K \sum_{l=0}^\infty \sum_{j=1}^M \|\widetilde{\psi}_{l,j} f\|_{r,s,I_l\times \omega_j,\ast}^2\ ,$$ 
which shows \eqref{localizationineq}.
\end{proof}
\textcolor{black}{Note the following Friedrichs inequality, which follows from the time-independent Friedrichs inequality and Fourier transformation into the time domain:}
$$\|\widetilde{\psi}_{l,j} f\|_{r,1,\Gamma'}\lesssim _\sigma  \|\nabla( \widetilde{\psi}_{l,j} f)\|_{r,0,\Gamma'} +  \|\partial_t(\widetilde{\psi}_{l,j} f)\|_{r,0,\Gamma'}\ .$$
\textcolor{black}{Therefore} \begin{align*}\|\widetilde{\psi}_{l,j} f\|_{r,s,I_l\times \omega_j,\ast}^2 &\lesssim \|\widetilde{\psi}_{l,j} f\|_{r,0,\Gamma'}^{2(1-s)}\|\widetilde{\psi}_{l,j} f\|_{r,1,\Gamma'}^{2s}\\ &\lesssim_\sigma d_j^{2(1-s)}(1+d_j^2)^s\left( \|\nabla(\widetilde{\psi}_{l,j} f)\|_{r,0,\Gamma'}^2 +  \|\partial_t(\widetilde{\psi}_{l,j} f)\|_{r,0,\Gamma'}^2\right)\ .\end{align*}
Here $d_j = \max\{\delta_j, \Delta t\}$, where $\delta_j$ is the width of the support of $\phi_j$, defined as the smallest number such that the following is true: There exists a direction $n \in \mathbb{R}^3$, $|n|=1$, such that for all $x \in \mathbb{R}^3$ and for each plane $H$ perpendicular to $n$ with $x \in H$, the intersection $\mathrm{supp}\ \phi_j \cap H$ is a Lipschitz curve of length $\leq \delta_j$. 

We now use  $\mathcal{V}\partial_t (\phi-\phi_{h, \Delta t}) = \mathcal{R}$, the coercivity of $\mathcal{V}\partial_t$ in Proposition \ref{DPbounds}, and Theorem \ref{cct31} with $s=\frac{1}{2}$. We obtain the following a posteriori error estimate, with the residual localized in the space-time elements:
\begin{corollary}\label{cccor}  Let $\phi  \in H^{1}_\sigma(\mathbb{R}^+,H^{-\frac{1}{2}}(\Gamma))$ be the solution to \eqref{DP}, and let $\phi_{h,\Delta t} \in H^{1}_\sigma(\mathbb{R}^+,H^{-\frac{1}{2}}(\Gamma))$ such that $\mathcal{R} = \partial_t f - \mathcal{V} \partial_t \phi_{h,\Delta t}\in H^{0}_\sigma(\mathbb{R}^+,H^1(\Gamma))$. Then with \textcolor{black}{$\Box_{l,j} = \mathrm{supp}\ \widetilde{\psi}_{l,j}$},
$$\|\phi-\phi_{h, \Delta t}\|^2_{0, -\frac{1}{2}, \Gamma, \ast} \lesssim_\sigma \sum_{l,j}d_j \left(\|\nabla \mathcal{R}\|_{0,0,\textcolor{black}{\Box_{l,j}}}^2 + \|\partial_t \mathcal{R}\|_{0,0, \textcolor{black}{\Box_{l,j}}}^2\right) \simeq \sum_{l, j} \max\{(\Delta t)_l, h_{\textcolor{black}{l,j}}\}\ \|\mathcal{R}\|_{0,1,  \textcolor{black}{\Box_{l,j}}}^2\ .$$
\end{corollary}

\section{Lower bounds} \label{lower}

As for time-independent problems, for the discussion of lower bounds we restrict ourselves to globally quasi-uniform meshes on a polyhedral screen $\Gamma$.

 Because of the different norms in the upper and lower bounds for $B_{D}$ in Proposition \ref{DPbounds}, the a posteriori estimate only satisfies a weak variant of efficiency: \\ \textcolor{black}{Provided $\phi \in H^{2}_\sigma(\mathbb{R}^+,H^{0}(\Gamma))$ from the mapping properties of $\mathcal{V}$ in Theorem  \ref{mapthm2} we conclude for $\varepsilon\in (0,1)$:} 
\begin{align*}&\max\{\Delta t, h\}^{-\frac{1-\varepsilon}{2}}\ \|\phi-\phi_{h,\Delta t}\|_{{0,-\frac{1}{2},\Gamma,\ast}}\lesssim\\&\quad \qquad \|\mathcal{R}\|_{0,1-\varepsilon,\Gamma}=\|\mathcal{V}(\partial_t \phi-\partial_t \phi_{h,\Delta t})\|_{0,1-\varepsilon, \Gamma} \lesssim \|\phi-\phi_{h,\Delta t}\|_{2,-\varepsilon,\Gamma}\leq \|\phi-\phi_{h,\Delta t}\|_{2,0,\Gamma}\ .\end{align*} \textcolor{black}{A proof of the sharp estimate, $\varepsilon =0$, follows from the endpoint estimate $s= \frac{1}{2}$ in Theorem \ref{mapthm2}. As mentioned there, it is known for $\Gamma$ of class $C^{1,\alpha}$ and, for elliptic problems, from a deep result by Verchota for Lipschitz $\Gamma$. } \\

As in the elliptic case, we aim to use the mapping properties of $\mathcal{V}$ together with approximation properties of the finite element spaces to recover the same spatial Sobolev index $-\frac{1}{2}$ in the upper and lower estimates.\\

\begin{theorem}\label{lowerboundthm}
Assume that $\mathcal{R} \textcolor{black}{ = \partial_t{f} - \mathcal{V} \partial_t{\phi}_{h,\Delta t}}\in H^{0}_\sigma(\mathbb{R}^+, H^{1}(\Gamma))$ and that the ansatz functions {$V=W^p_{h,\Delta t}  \subset H^{2}_\sigma(\mathbb{R}^+,H^{0}(\Gamma)) \cap V_{h,\Delta t}^{p,q}$} satisfy
\begin{equation} \label{star}
    \inf_{\psi_{h \Delta t} \in V} \Vert \phi- \psi_{h \Delta t} \Vert_{2,0,\Gamma} \simeq \max\{ h, \Delta t \}^{\beta}  
 \end{equation}
for some $\beta>0$.  Then for all $\varepsilon\in (0,1)$
\begin{align*}
    \Vert \mathcal{R} \Vert_{0,1-\varepsilon,\Gamma} \lesssim & \max \{ h^{-\frac{1}{2}}, (\Delta t)^{-\frac{1}{2}} \} \Vert \phi - \phi_{h \Delta t} \Vert_{2,-\frac{1}{2},\Gamma} .
\end{align*}
\end{theorem}
\begin{rem}
 In Theorem \ref{starthm} below, the hypothesis \eqref{star} is verified using the singular expansion of the solution $\phi$ at the edges and corners. 
\end{rem}

\textcolor{black}{For the proof, we require an auxiliary projection $P_{h, \Delta t} $ from $ H^{2}(\mathbb{R}^+,H^{-1}(\Gamma))$ into the space of ansatz functions $\subseteq H^{2}(\mathbb{R}^+,L^{2}(\Gamma))$ such that $P_{h, \Delta t} $ is $H^{2}(\mathbb{R}^+, L^{2}(\Gamma))$-stable.}

\textcolor{black}{For the construction of $P_{h, \Delta t } = P_{\Delta t} \mathcal{Q}_h^\ast$ we use a  projection $P_{\Delta t}$ on the space of cubic splines $W_{\Delta t}\subset V^3_{\Delta t}$ in time and a Galerkin projection $\mathcal{Q}_h^\ast$ in space. }

\textcolor{black}{More precisely, we consider the projection  $\mathcal{Q}_h : L^{2}(\Gamma) \to V^p_h$ defined as follows: For $\phi \in L^{2}(\Gamma)$, we let $\mathcal{Q}_h \phi \in V^p_h$ be the unique solution to the  variational problem $$\langle \mathcal{Q}_h \phi ,\psi_h \rangle =  \langle \phi ,\psi_h \rangle$$
for all $\psi_h \in V^p_h$. $\mathcal{Q}_h$ is bounded on $L^{2}(\Gamma)$ and self-adjoint, $\mathcal{Q}_h = \mathcal{Q}_h^\ast$. Therefore, if  $\mathcal{Q}_h$ restricts to a bounded operator on $H^s_\omega(\Gamma)$, $\mathcal{Q}_h = \mathcal{Q}_h^\ast$ extends with the same norm to a bounded operator on $\widetilde{H}^{-s}_\omega(\Gamma)$. }

\textcolor{black}{For triangulations on which $\mathcal{Q}_h$ is a bounded operator on $H^1(\Gamma)$, we show that the  operator norm is uniformly bounded in $\omega$:
\begin{lemma} Let $s\in [0,1]$ and $\mathcal{Q}_h$ bounded on $H^1(\Gamma)$. Then for all $\phi \in H^s_\omega(\Gamma)$, we have:
$$\|\mathcal{Q}_h \phi\|_{s,\omega, \Gamma} \leq C \|\phi\|_{s,\omega, \Gamma}.$$
For all $\phi \in \widetilde{H}^{-s}_\omega(\Gamma)$, we have:
$$\|\mathcal{Q}_h \phi\|_{-s,\omega, \Gamma, \ast} \leq C \|\phi\|_{-s,\omega, \Gamma,\ast}.$$
\end{lemma}
The conditions are satisfied, for example, on quasi-uniform meshes \cite{cch1}, or more generally on adaptive meshes generated by NVB refinements. We only use it in the quasi-uniform case.
}

\begin{proof}
\textcolor{black}{For the first assertion we note that it is clear for $s=0$. We show it for $s=1$, and the case for general $s\in [0,1]$ follows from interpolation.}\\

\textcolor{black}{
By assumption, 
$$\|\nabla \mathcal{Q}_h \phi\|_{L^2(\Gamma)} \leq C \|\phi\|_{H^1(\Gamma)}.$$ Further $\mathcal{Q}_h$ is bounded with norm $1$ on $L^{2}(\Gamma)$, so that 
$$\|\mathcal{Q}_h \phi\|_{L^2(\Gamma)} \leq \|\phi\|_{L^2(\Gamma)}\ .$$
With
$$\|\mathcal{Q}_h \phi\|_{1,\omega, \Gamma}^2 = \|\omega \mathcal{Q}_h \phi\|_{L^2(\Gamma)}^2+\|\nabla \mathcal{Q}_h \phi\|_{L^2(\Gamma)}^2, $$
we conclude $$\|\mathcal{Q}_h \phi\|_{1,\omega, \Gamma}^2 \leq \|\omega \phi\|_{L^2(\Gamma)}^2 + C^2 \|\phi\|_{H^1(\Gamma)}^2 \leq C^2 \|\phi\|_{1,\omega, \Gamma}^2\ .$$
This shows the first assertion for $s=1$.
}

\textcolor{black}{The second assertion follows from the first one, with the same constant $C$. Indeed, 
\begin{align*}
\|\mathcal{Q}_h \phi\|_{-s,\omega, \Gamma, \ast} &= \sup_{0 \neq \psi \in H^s_\omega(\Gamma)}  \frac{\langle\mathcal{Q}_h \phi , \psi\rangle_{L^2(\Gamma)}}{\|\psi\|_{s,\omega, \Gamma}}= \sup_{0 \neq \psi \in H^s_\omega(\Gamma)}  \frac{\langle \phi , \mathcal{Q}_h\psi\rangle_{L^2(\Gamma)}}{\|\psi\|_{s,\omega, \Gamma}} \leq C\|\phi\|_{-s,\omega, \Gamma, \ast}\ .
\end{align*}}
\end{proof}
\textcolor{black}{Using the Laplace transform to transfer the result into the time-domain, we obtain $$\|\mathcal{Q}_h \phi\|_{r,-s, \Gamma, \ast} \leq C \|\phi\|_{r,-s, \Gamma,\ast}$$ for $s \in [0,1]$ and all $r$. {Further note that the interpolation operator $P_{\Delta t}$ in time for cubic splines is continuous on $H^2$ functions.} The continuity extends to $ H^{2}(\mathbb{R}^+,\widetilde{H}^{-s}(\Gamma))$. We conclude:
\begin{lemma} \label{pproj} Consider a quasi-uniform mesh, $s\in [0,1]$ and $P_{h, \Delta t } = P_{\Delta t} \mathcal{Q}_h$.
For all $\phi \in  H^{2}(\mathbb{R}^+,\widetilde{H}^{-s}(\Gamma))$, we have
$$\|P_{h, \Delta t } \phi\|_{2,-s, \Gamma, \ast} \leq C \|\phi\|_{2,-s, \Gamma,\ast}.$$
\end{lemma}
The approximation properties of $P_{h, \Delta t }$ analogous to Lemma \ref{interp} follow as in \cite{Glaefke}, Proposition 3.54.}

\begin{proof}[Proof of Theorem \ref{lowerboundthm}]The theorem partly follows the functional analytic approach of \cite{cc} for the time-independent case. We consider the following quantities:
the error of the best approximation $\Pi_{h, \Delta t} \phi $ of $ \phi $ in $ H^{2}_\sigma(\mathbb{R}^+,L^{2}(\Gamma))$,
	$$
	  E(\phi,h , \Delta t) = \Vert \phi - \Pi_{h, \Delta t} \phi  \Vert_{2, 0,\Gamma}\ ,
	$$
where $E(\phi,h , \Delta t) \simeq \max \{ h, \Delta t \}^{\beta}$ by \eqref{star}; the relative error of best approximation,
$$
	  F(\phi, h, \Delta  t) = \frac{\Vert \phi - \Pi_{h, \Delta t} \phi  \Vert_{2,-1,\Gamma}}{\Vert \phi - \Pi_{h, \Delta t} \phi  \Vert_{2,0,\Gamma}};
	$$
and the inverse inequality \cite{Glaefke}
	$$
	  G(h,\Delta t) = \sup_{\psi_{h \Delta t} \neq 0} \frac{\Vert \psi_{h, \Delta t} \Vert_{2,0,\Gamma}}{ \Vert \psi_{h, \Delta t} \Vert_{2,-1,\Gamma}} \lesssim \max\{h^{-1}, \Delta t^{-1}\}.
	$$
Their analysis relies on the projection operator $P_{h, \Delta t}$ from Lemma \ref{pproj}.

Note that $F(\phi, h, \Delta t) \lesssim \max\{ h , \Delta t \}$. Indeed because $\phi \in H^{2}_\sigma(\mathbb{R}^+,H^{0}(\Gamma))$, by duality and the approximation properties of $ \Pi_{h, \Delta t}$:
\begin{align*}
    \Vert \phi - \Pi_{h,\Delta t} \phi \Vert_{2,-1,\Gamma} &= \sup_{\eta} \frac{\langle \phi - \Pi_{h, \Delta t} \phi, \eta - \Pi_{h, \Delta t} \eta\rangle}{\Vert \eta \Vert_{-2,1,\Gamma}}\\
    & \lesssim \max\{ h , \Delta t \} \Vert \phi - \Pi_{h, \Delta t} \phi \Vert_{2,0,\Gamma}.
\end{align*}
This uses the estimate $\| \eta - \Pi_{h, \Delta t} \eta \|_{-2,0,\Gamma} \lesssim \max\{ h , \Delta t \} \|\eta\|_{-2,1,\Gamma}$.

From the beginning of this section, we recall $$\|\mathcal{R}\|_{0,1-\varepsilon,\Gamma}\lesssim \|\phi-\phi_{h,\Delta t}\|_{2,0,\Gamma}\leq \|\phi-\Pi_{H, \Delta T} \phi\|_{2,0,\Gamma} +  \|\Pi_{H, \Delta T} \phi-\phi_{h,\Delta t}\|_{2,0,\Gamma}\ .$$ 
Note that the first term is $\|\phi-\Pi_{H, \Delta T} \phi\|_{2,0,\Gamma} = E(\phi, H, \Delta T)$. For the second term we observe that 
\begin{align*}
\|\Pi_{H, \Delta T} \phi-\phi_{h,\Delta t}\|_{2,0, \Gamma} \lesssim G(H, \Delta T) \|\Pi_{H, \Delta T} \phi-\phi_{h,\Delta t}\|_{2,-1, \Gamma}\ ,
\end{align*}  
so that $\|\Pi_{H, \Delta T} \phi-\phi_{h,\Delta t}\|_{2,0, \Gamma} \lesssim G(H, \Delta T)^{1/2} \|\Pi_{H, \Delta T} \phi-\phi_{h,\Delta t}\|_{2,-\frac{1}{2}, \Gamma}$.  We conclude 
$$\|\mathcal{R}\|_{0,1-\varepsilon,\Gamma}\lesssim E(\phi, H, \Delta T) +  G(H, \Delta T)^{1/2} \|\Pi_{H, \Delta T} \phi-\phi_{h,\Delta t}\|_{2,-\frac{1}{2}, \Gamma}\ .$$
Further, $$\|\Pi_{H, \Delta T} \phi-\phi_{h,\Delta t}\|_{2,-\frac{1}{2}, \Gamma} \leq \|\Pi_{H, \Delta T} \phi- \phi\|_{2,-\frac{1}{2}, \Gamma} + \|\phi-\phi_{h,\Delta t}\|_{2,-\frac{1}{2}, \Gamma}\ .$$
From the definition of $F$ and interpolation, $\|\Pi_{H, \Delta T} \phi- \phi\|_{2,-\frac{1}{2}, \Gamma}$ is bounded by $$\lesssim F(\phi, H,\Delta T)^{1/2}\|\Pi_{H, \Delta T} \phi- \phi\|_{2,0, \Gamma}  = F(\phi, H,\Delta T)^{1/2}E(\phi, H, \Delta T)\ .$$
To sum up, 
\begin{align*}
\|\phi-\phi_{h,\Delta t}\|_{2,0,\Gamma} & \lesssim E(\phi, H, \Delta T) + G(H, \Delta T)^{1/2} F(\phi, H,\Delta T)^{1/2}E(\phi, h, \Delta t) \\ & \qquad +  G(H, \Delta T)^{1/2} \|\phi-\phi_{h,\Delta t}\|_{2,-\frac{1}{2}, \Gamma}\\ &
\leq \frac{E(\phi, H, \Delta T)}{E(\phi, h, \Delta t)} \|\phi-\phi_{h,\Delta t}\|_{2,0,\Gamma} (1+ G(H, \Delta T)^{1/2} F(\phi, H,\Delta T)^{1/2})\\ & \qquad +  G(H, \Delta T)^{1/2} \|\phi-\phi_{h,\Delta t}\|_{2,-\frac{1}{2}, \Gamma} \ ,
\end{align*}
or, with $\delta = \frac{E(\phi, H, \Delta T)}{E(\phi, h, \Delta t)} (1+ F(\phi, H,\Delta T)^{1/2} G(\phi, H,\Delta T)^{1/2})$ and a constant $C$, 
$$\|\phi-\phi_{h,\Delta t}\|_{2,0,\Gamma} \lesssim \frac{G(H, \Delta T)^{1/2}}{1- C\delta} \|\phi-\phi_{h,\Delta t}\|_{2,-\frac{1}{2}, \Gamma}\ .$$
If $\delta= \frac{E(\phi, H, \Delta T)}{E(\phi, h, \Delta t)} (1+ F(\phi, H,\Delta T)^{1/2} G(\phi, H,\Delta T)^{1/2})<\frac{1}{2C}$, we obtain 
  $$
    \Vert \mathcal{R} \Vert_{0,1-\varepsilon, \Gamma} \lesssim G(H,\Delta T)^{1/2} \|\phi - \phi_{h \Delta t} \|_{2,-\frac{1}{2}, \Gamma}\ .
  $$
 Here $H$ and $\Delta T$ are sufficiently small compared to $h$ and $\Delta t$, and it remains to choose them so that $\delta<\frac{1}{2C}$. Set $H=\rho h$ and $\Delta T = \rho \Delta t$ for $\rho \in (0,1)$. Using that $$F(\phi, \rho h,\rho \Delta t) G(\phi, \rho h, \rho \Delta t) \simeq \frac{\max\{h,\Delta t\}}{\min\{h,\Delta t\}}$$ is uniformly bounded in $\rho \in (0,1)$, so that it suffices to show that  $\frac{E(\phi, \rho h, \rho \Delta t)}{E(\phi, h, \Delta t)} \to 0$ as $\rho$ tends to $0$. This follows from \eqref{star}.
\end{proof}

\section{Best approximation and lower bounds}\label{sec:2D-Approx}

In this section \textcolor{black}{we} verify the hypothesis \eqref{star} in Theorem \ref{lowerboundthm} for polyhedral domains, by proving upper and lower bounds for the best approximation of the solution to the wave equation with Dirichlet boundary conditions. 

Let $\Omega$ be a polyhedral domain and $u$ a solution to the wave equation in $\Omega$:
\begin{align}
\partial_t^2 u(t,x) -\Delta u(t,x)&=0 &\quad &\text{in }\mathbb{R}^+_t \times \Omega_x\ , \\
u(t,x)&=g(t,x) &\quad &\text{on }  \Gamma= \partial \Omega \ ,\\
u(0,x)=\partial_t u(0,x)&=0 &\quad &\text{in } \Omega .
\end{align}
The function $u$ exhibits well-known singularities at non-smooth boundary points of the domain. Locally near an edge or a corner, $\Omega$ is of the form $\mathbb{R}_+ \times \mathcal{K}$, where the base $\mathcal{K} \subset S^2$ is a smooth or polygonal subset of the sphere. The solution may be decomposed into a leading part given by explicit singular functions plus less singular terms \cite{hp,kokotov, kokotov3, plamenevskii}. We refer to \cite[Theorem 7.4 and Remark 7.5]{kokotov} for details in the case of the Neumann problem in a wedge, respectively \cite[Theorem 4.1]{kokotov3} for the Dirichlet problem in a cone and state the decomposition in terms of polar coordinates $(r,\theta)$ centered at the vertex $(0,0,0)$: 
\begin{align}\nonumber
u(t,x) &=u_{0}(t,r, \theta) + \chi(r)r^{\lambda} a(t,\theta) +\tilde{\chi}(\theta)b_{1}(t,r)(\sin(\theta))^{\nu}\\& \qquad + \tilde{\chi}(\textstyle{\frac{\pi}{2}}-\theta)b_{2}(t,r)(\cos(\theta))^{\nu} \ ,\\
\nonumber
\partial_n u (t,x)&=\psi_{0}(t,r, \theta) + \chi(r)r^{\lambda-1} a(t,\theta) +\tilde{\chi}(\theta)b_{1}(t,r)r^{-1}(\sin(\theta))^{-\nu}\\& \qquad + \tilde{\chi}(\textstyle{\frac{\pi}{2}}-\theta)b_{2}(t,r)r^{-1}(\cos(\theta))^{-\nu} \ . \label{decomposition}
\end{align}
Here, $\nu = \frac{\pi}{\alpha}$, where $\alpha$ is the opening angle of the wedge, and $\lambda = -\frac{1}{2}+\sqrt{\frac{1}{4} + \mu}$, where $\mu$ is the smallest eigenvalue of the Laplace-Beltrami operator with Dirichlet boundary conditions  in the subdomain $\mathcal{K}$ of the sphere. $\chi$, $\tilde{\chi}$ are cut-off functions and $a, b_{j}$ sufficiently regular.
For generic problems the functions $a, b_1, b_2$ are not identically zero.

From the representation formula, the wave equation translates into the boundary integral equation $V \phi = f$, with $f=(1-\mathcal{K}) g$ and solution $\phi = \partial_n u|_\Gamma$.

The main theorem concerning the approximation of $\phi$  is:
\begin{theorem}\label{starthm}
Assume that the coefficient functions $a, b_1, b_2$ are not identically $0$. Then $E(\phi, h, \Delta t) \simeq \max\{h, \Delta t\}^{\max\{\nu-\frac{1}{2}, \lambda\}}$. 
\end{theorem}
In particular, hypothesis \eqref{star} is satisfied. A similar result in the elliptic case was known if $\Gamma$ is a curve, i.e. in dimension $2$ \cite{cc}.

The key step in the proof of Theorem \ref{starthm} is to show the result for bilinear basis functions on a rectangular mesh.  For simplicity of notation, we restrict to one boundary face of $\Gamma$ and assume it is given by $Q \times \{0\}$ with corner of the domain at $(0,0,0)$, where $Q=(0,1)^2=\bigcup_{kl} Q_{kl}$, $Q_{kl} = [x_{k-1}, x_k)\times [x_{l-1}, x_l)$ and $x_k = kh$. 

Recall the following estimate from \cite{graded}:
\begin{lemma}\label{lemma3.4}
	Let $-1\leq s\leq 0,\; {0\leq r\leq \rho \leq p+1},\; Q=[0,h_1]\times[0,h_2],\; u\in H^{\rho}_\sigma([0,\Delta t], H^1(R))$, $\Pi_t^{{p}} u$ the orthogonal projection onto piecewise polynomials in $t$ of order {$p$}, and $\Pi_{x,y}^0$ the orthogonal projection onto piecewise constant polynomials in space, $\Pi_{x,y}^0 u=\frac{1}{h_1 h_2}\int\limits_Q u(t,x,y)  dy\, dx$. Then for ${U} =  \Pi_t^{{p}} \Pi^0_{x,y} u$ we have 
	\begin{align}\label{3.22}
		\| u-{U}\|_{r,s,Q,\ast}&\lesssim (\Delta t)^{\rho-r}{\max\{h_1,h_2,\Delta t \}^{-s}}\|\partial_t^\rho u\|_{L^2([0,\Delta t]\times Q)} \\ & \qquad+  \max\{h_1,h_2,\Delta t \}^{-s}\left(h_1 \| u_x\|_{L^2([0,\Delta t] \times Q)}  + h_2 \| u_y\|_{L^2([0,\Delta t] \times Q)} \right)\ . \nonumber
	\end{align}
	If $u(t,x,y) =u_1(t,x)u_2(y),\; u_1\in H^{\rho}_\sigma( [0,\Delta t], H^1([0,h_1])), \; u_2\in H^1( [0,h_2])$ then 
	\begin{align*}
		\| u-{U}\|_{r,s,Q,\ast}& \lesssim (\Delta t)^{\rho-r}{\max\{h_1,\Delta t \}^{-s}}\|\partial_t^\rho u\|_{L^2([0,\Delta t]\times Q)} \\ & \qquad+\left(h_1^{1-s}\| u_x\|_{{L^2}([0,\Delta t] \times Q)} + h_2^{1-s}\| u_y\|_{{L^2}([0,\Delta t] \times Q)} \right)\ .
	\end{align*}
\end{lemma}

\begin{proof}[Proof of Theorem \ref{starthm}] We use the upper bound $E(\phi, h, \Delta t) \lesssim \max\{h, \Delta t\}^{\max\{\nu-\frac{1}{2}, \lambda\}}$ and first approximate the corner singularity $f=r^{\lambda-1}a(t,\theta)$, $a \in H^\rho_\sigma(\mathbb{R}^+, H^1([0,\pi/2]))$. With $\left. p\right|_{Q_{kl}}= \Pi_t^{{p}} \Pi^0_{x,y} f$ and Lemma \ref{lemma3.4} one obtains
\begin{align}\label{3.23}
	\qquad \Vert f-p\Vert^2_{0,0,Q}\lesssim \sum\limits^N_{k,l=1\atop k+l\neq 2} &\left((\Delta t)^{2\rho}\|\partial_t^\rho f\|_{0,0,Q_{kl}}^2+h_k^2 \Vert f_x\Vert^2_{0,0,Q_{kl}} + h_l^2 \Vert f_y\Vert^2_{0,0,Q_{kl}} \right) \nonumber \\ &+ \Vert f-p\Vert_{0,0,Q_{11}}^{2}\ .
\end{align}
For $k\geqslant 2,\; l\geqslant 2$ the following estimate holds, with the zeta function $\zeta(s)$: 
\begin{align*}
	\sum_{k,l=2}^{N} h_k^2\Vert f_x\Vert^2_{0,0,Q_{kl}} &= \sum_{k=2}^{N} h_{k}^{2} \int_0^\infty \int_{\phi=0}^{\pi/2} \tilde{w}(t,\theta)^{2} d\theta\ d_\sigma t \int_{\sqrt{2}kh}^{\sqrt{2}(k+1)h} r^{2  \lambda-4} r dr \\ &\lesssim  \sum_{k=2}^{\infty} h^{2  \lambda} k^{2  \lambda -3} = c  h^{2  \lambda} \zeta(3- 2  \lambda) \ .
\end{align*}
Here we have used 
\begin{equation*}
 | f_{x}(t,x,y) | \leqslant r^{  \lambda-2} \tilde{w}(t, \theta)
\end{equation*}
with $\tilde{w} \in H^\rho_\sigma(\mathbb{R}^+, H^0([0, \pi/2]))$. \\

For $k=1,\, l>1$ (and analogously for $k>1$, $l=1$), in \eqref{3.23} we compute that
\begin{align*}
	\sum\limits_{l=2}^N h_1^2\Vert f_x\Vert^2_{0,0,Q_{1l}} 
		&\leqslant h_{1}^2\int_0^\infty \int\limits_{x=0}^{h_1}\int\limits_{y=0}^1\vert f_x(t,x,y)\vert^2 dy \, dx \, d_\sigma t\\
		& \leqslant h_{1}^2\int_0^\infty \int\limits_{r=h}^{\sqrt{2}} \int\limits_{\phi=0}^{\pi/2} a^{2}(t,\theta) d\theta\ r^{2 \lambda -3} dr \ d_\sigma t\\
	& \lesssim h_{1}^{2} + h_{1}^{2  \lambda}\ .
\end{align*}
Finally, in the corner $k=1$, $l=1$ the singular function $f\in H^\rho_\sigma(\mathbb{R}^+, H^0(Q_{11}))$, because $\lambda >0$. 
Now
\begin{equation*}
	\Vert f-p\Vert_{0,0,Q_{11}} \lesssim \Vert f\Vert_{0,0,Q_{11}} \simeq h_1^{\lambda} \ .
\end{equation*}

We now consider the approximation of the edge singularities, which are of two types, fixing $Q=[0,1]^2$ with singularity at the $x$-axis:
\begin{enumerate}
\item[(i)] $f(t,x,y)=\chi(x)b(t,x) y^{\nu-1}$ with  edge intensity factor  $b\in H^\rho_\sigma(\mathbb{R}^+, H_0^1(\mathbb{R}^+))$, 
\item[(ii)] $f(t,x,y) = \chi(x)b(t)x^{ \lambda - \nu}y^{\nu-1}$ with a corner singularity in the edge intensity factor.
\end{enumerate} 
We have
\begin{equation}\label{3.25}
	\Vert f-p\Vert^2_{0,0,Q} \lesssim \sum\limits_{k,l=1}^N \Vert f-p\Vert^2_{0,0,Q_{kl}}\ .
\end{equation}
First consider case (ii), where the time dependence factors out. Define 
\begin{equation*}
 p_{2}= { \frac{1}{h_{2}}} \int_{I_{j-1}^{*}} f_{2}(y) dy, \quad p_{1}=\frac{1}{h_{1}} \int_{I_{j}} f_{1}(x) dx,
 \quad f_{2}(y)=y^{\nu-1}, \quad f_{1} = x^{ \lambda -\nu}, 
\end{equation*}
where $h=h_{1}=h_{2}$, $ I_{j}=[x_{j-1},x_{j}]$ and $ I_{k}^{*}=[0,x_{k}]$. Then one computes
\begin{align*}
 \lVert f_{2}- p_{2} \rVert_{L^2 (0,h)}^{2} &= \frac{(\nu-1)^{2}}{(1+(\nu-1)^2)(2 (\nu-1)+1)} h^{2(\nu-1)+1} \\
 \lVert f_{2}-p_{2} \rVert_{L^{2}(a,a+h)}^2 &= a^{2(\nu-1)+1} \eta({\textstyle\frac{h}{a}}) , a > 0 , \text{ where}\\
 \eta(\delta)&= \frac{(1+\delta)^{2 (\nu-1) +1}-1}{2 (\nu-1)+1} - \frac{[(1+\delta)^{(\nu-1)}-1]^{2}}{\delta (1+(\nu-1))^{2}}  
\end{align*}
and $\delta >0$. With $ Q_{j}^{*}=\cup_{l=1}^{j-1} Q_{jl} = I_{j} \times I_{j-1}^{*}$, one notes
\begin{align*}
 \lVert f_{2}-p_{2} \rVert_{L^{2}(I_{j-1}^{*})}^{2} \lVert p_{1} \rVert_{L^{2}(I_{j})}^{2}
 &= h^{2 (\nu-1) +1} \Big[ \frac{(\nu-1)^{2}}{(1+(\nu-1)^2)(2 (\nu-1)+1)} + \\ & \sum_{l=1}^{j-1} 
 ({\textstyle\frac{x_{l}}{h}})^{2 (\nu-1)+1} \eta({\textstyle\frac{h_{2}}{x_{l}}})\Big] \int_{x_{j-1}}^{x_{j}} \frac{1}{h_{1}^{2}} 
 \left(\int_{x_{j-1}}^{x_j} x^{ \lambda-\nu} dx\right)^{2} dx
\end{align*}
and
\begin{equation*}
 \int_{x_{j-1}}^{x_{j}} \frac{1}{h_{1}^{2}}\left(\int_{x_{j-1}}^{x_j} x^{ \lambda-\nu} dx\right)^{2} dx = x_{j-1}^{2  \lambda-2 \nu+1} 
 \frac{x_{j-1}}{h} \frac{[(1+ \frac{h}{x_{j-1}})^{ \lambda-\nu+1}-1]^{2}}{( \lambda-\nu+1)^{2}}\ .
\end{equation*}
As $ \eta({\textstyle\frac{h}{x_{l}}})\lesssim{\textstyle\frac{h^{3}}{x_{l}^{3}}}$ and
\begin{align*}
 \frac{1}{\delta} \frac{[(1+ \delta)^{ \lambda-\nu+1}-1]^{2}}{( \lambda-\nu+1)^{2}}
 &= \frac{1}{\delta}\frac{[( \lambda-\nu+1) \delta+( \lambda-\nu+1)( \lambda-\nu) \delta^{2}+O(\delta^{3})]}{( \lambda-\nu+1)^{2}} \\
 &= \delta + ( \lambda-\nu+1) \delta^{2} + O(\delta^{3})\ ,
\end{align*}
we obtain
\begin{align*}
 &\sum_{j=2}^{N} \lVert f_{2}-p_{2} \rVert_{L^{2}(I_{j-1}^{*})}^{2} \lVert p_{1} \rVert_{L^{2}(I_{j})}^{2}
 = \sum_{j=2}^{N} h^{2 (\nu-1) +1 } [\frac{(\nu-1)^{2}}{(1+(\nu-1)^2)(2 (\nu-1)+1)} \\ &+ 
 \sum_{l=1}^{j-1} ({\textstyle \frac{x_{l}}{h}})^{2 (\nu-1) -2}] [(j-1)h]^{2  \lambda- 2 \nu+1} 
 \frac{(j-1)[(1+\frac{1}{j-1})^{ \lambda-\nu+1}-1]^{2}}{( \lambda - \nu +1)^{2}}\ .
\end{align*}
The estimate
\begin{align*}
 \sum_{j=2}^{N} \lVert f_{2}-p_{2} \rVert_{L^{2}(I_{j-1}^{*})}^{2} \lVert p_{1} \rVert_{L^{2}(I_{j})}^{2}
 &= c \sum_{j=2}^{N} h^{2 (\nu-1)+1} (j-1)^{2  \lambda-2 \nu} h^{2  \lambda -2 \nu+1} \\
 &=  h^{2 (\nu-1)+1} h^{2  \lambda -2 \nu+1} \sum_{k=1}^{N-1} k^{2  \lambda-2 \nu} 
 \lesssim \zeta(2\nu-2\lambda) h^{2  \lambda}
\end{align*}
follows and may be integrated in time.\\
Now consider
\begin{align*}
\lVert f_{1}-p_{1} \rVert_{L^{2}(I_{j})}^{2}  &= \lVert x^{ \lambda-\nu} - 
{\textstyle\frac{1}{h} \int_{I_{j}} x^{ \lambda -\nu} dx} \rVert_{L^{2}(I_{j})}^{2} \\
&= \begin{cases}
     a^{2 ( \lambda -\nu) +1} \eta(\frac{h}{a}), & a > 0 \\
     \frac{( \lambda -\nu)}{(1+( \lambda -\nu)^{2})(2( \lambda -\nu)+1)}
     h^{2 ( \lambda -\nu)+1} , & a=0\ .
  \end{cases}
\end{align*}
For $ j \geq 2$
\begin{align*}
 \lVert f_{1}-p_{1} \rVert_{L^{2}(I_{j})}^{2} = {\textstyle\frac{x_{j}^{2( \lambda -\nu)+1}}{h^{2( \lambda -\nu)+1}}}
 \eta\left({\textstyle\frac{h}{x_{j}}}\right) h^{2( \lambda -\nu)+1} 
 = ({\textstyle \frac{x_{j}}{h}})^{2( \lambda -\nu)-2} h^{2( \lambda -\nu)+1}, 
\end{align*}
\begin{equation*}
 \lVert f_{2} \rVert_{L^{2}(I_{j-1}^{*})}^{2} = \int_{0}^{x_{j-1}} y^{2 \nu -2} dy \lesssim x_{j-1}^{2 \nu-1},
\end{equation*}
and
\begin{align*}
 \lVert f_{1}-p_{1} \rVert_{L^{2}(I_{j})}^{2} \lVert p_{2} \rVert_{L^{2}(I_{j-1}^{*})}^{2}
 &= ({\textstyle\frac{x_{j}}{h}})^{2 ( \lambda -\nu)-2} ({\textstyle\frac{x_{j-1}}{h}})^{2 \nu-1} 
 h^{2( \lambda -\nu)+1} h^{2 \nu-1} \\ &\leqslant 
 ({\textstyle\frac{x_{j}}{h}})^{2 ( \lambda -\nu)-2+2 \nu-1} h^{2 ( \lambda -\nu)+1 +2\nu-1}.
\end{align*}
We have
\begin{align*}
 \sum_{j=2}^{N} \lVert f_{1}-p_{1} \rVert_{L^{2}(I_{j})}^{2} \lVert f_{2} \rVert_{L^{2}(I_{j-1}^{*})}^{2}
 &= c \sum_{j=1}^{\infty} j^{2( \lambda -\nu)-2+2\nu-1} 
 h^{2 \lambda-2 \nu+1+2 \nu-1} = c \zeta(3-2  \lambda) h^{2  \lambda}\ .
\end{align*}
Again, this estimate may be integrated in time.\\
We now consider case (i): $f(t,x,y) = b(t,x) y^{\nu-1}+(\chi(x)-1)b(t,x)y^{\nu-1} =: f_{1}+f_{2}$ for $ \nu > \frac{1}{2}$ and 
$ Q=[0,1]^{2}=I \times I $, with $I=[0,1]$.
Again we define
\begin{equation*}
 q_{1} = \Pi^p_t \int_{0}^{1} b(t,x) dx , q_{2} = \int_{0}^{1} y^{\nu-1} dy \ .
\end{equation*}
Note that with $f_{2}(t,x,y)=(\chi(x)-1)b(t,x)y^{\nu-1} \in H^0_\sigma(\mathbb{R^+},H^{1}(Q))$ we get
\begin{equation*}
 \lVert f_{2}-q_{1}q_{2} \rVert_{0,0,Q}^{2} \lesssim h^{2} \ .
\end{equation*}
Since 
\begin{align*}
 \lVert y^{\nu-1} -q_{2} \rVert_{L^2(I)}^2 \simeq h^{2 \nu-1},
 \lVert y^{\nu-1} \rVert_{L^2(I)}^{2} \leq c 
\end{align*}
and
 $$\lVert b - q_{1} \rVert_{0,0,I}^2 \lesssim \max\{h,\Delta t\}^{2} \left(\lVert {\partial_t} b \rVert_{0,0,I}^{2}+ \lVert {\partial_x} b \rVert_{0,0,I}^{2}\right),$$
we have
\begin{align}
\nonumber &\lVert b(t,x) y^{\nu-1}-q_1(t,x) q_{2}(y) \rVert_{0,0,Q}^2 \\&\lesssim \lVert b \rVert_{0,0,I}^{2}
 \lVert y^{\nu-1}-q_{2} \rVert_{L^{2}(I)}^{2} + \lVert y^{\nu-1} \rVert_{L^{2}(I)}^{2} 
 \lVert b - q_{1} \rVert_{0,0,I}^{2}\lesssim h^{2 \nu-1} + \max\{h,\Delta t\}^{2}  \ . \label{nuineq}
\end{align}

For the lower bound, we first consider the error in $Q_{11}$ resulting from the corner singularity. There the error of approximation by a spatially constant function $c$ is given by
$$\|a(t, \theta)r^{\lambda-1}-c(t)\|_{0, 0, Q_{11}}^2 \gtrsim \int_0^\infty \int_{0}^h \int_{0}^{\pi/2}r (a(t, \theta)r^{\lambda-1}-c)^2 d\theta\ dr\ d_\sigma t\ .$$ If $a \neq 0$, we may find a small intervall $(\theta_0-\delta, \theta_0+\delta)$ on which $a$ is nonzero for a time interval $I$. We estimate
$$\|a(t,\theta)r^{\lambda-1}-c\|_{0,0, Q_{11}}^2  \gtrsim \int_I \int_{0}^h \int_{\theta_0-\delta}^{\theta_0+\delta}r (a(t,\theta)r^{\lambda-1}-c)^2 d\theta \ dr \ d_\sigma t\ .$$
As $a$ is a restriction of an eigenfunction of the Laplace-Beltrami operator, it is smooth, and up to higher order terms {(h.o.t.)} in $h$ we compute 
\begin{align*}
&\int_I \int_{0}^h \int_{\theta_0-\delta}^{\theta_0+\delta}r (a(t,\theta)r^{\lambda-1}-c)^2 d\theta \ dr\ d_\sigma t\\&= \int_I\int_{0}^h \int_{\theta_0-\delta}^{\theta_0+\delta}r(a(t,\theta_0)r^{\lambda-1}-c)^2 d\theta\ dr \ d_\sigma t  + h.o.t.\\
&=2 \delta \int_I \int_{0}^h r(a(t,\theta_0)r^{\lambda-1}-c)^2 dr \ d_\sigma t + h.o.t.
\end{align*}
Now we may explicitly compute the infimum over $c$:
\begin{align*}
\int_I \int_{0}^h r (a(t,\theta_0)r^{\lambda-1}-c)^2 dr &\geq C_\sigma |I| \min_{t \in I} a(t, \theta_0)^2\frac{(\lambda-1)^2}{2 \lambda (\lambda+1)^2} h^{2 \lambda}\ .
\end{align*}
This lower order bound for the error of order $h^\lambda$ matches with the upper bound from above.\\

An analogous argument for the regular edge intensity factor, case (i), shows that the inequality \eqref{nuineq} is an equality up to higher order terms in $h$. We therefore obtain 
$$ \lVert b(x) y^{\nu-1}-q_1(x) q_{2}(y) \rVert_{0,0,Q} \gtrsim  h^{\nu-\frac{1}{2}} + h.o.t.$$
In a neighborhood of the edge or corner, the solution is given by its singular expansion, and the approximation error coincides with the approximation error for the singular functions, $b(t,x) y^{\nu-1}$, $y_i^{ \lambda - \nu}\rho^{\nu-1}$, respectively $a(t,\theta)r^{\lambda-1}$, up to lower order terms in $h$. We conclude that 
$$E(\phi, h, \Delta) = \Vert \phi - \Pi_{h, \Delta t} \phi  \Vert_{2,0,\Gamma} \gtrsim \max\{h, \Delta t\}^{\min\{\nu-\frac{1}{2}, \lambda\}} + h.o.t.\ .$$  
References \cite{disspetersdorff, petersdorff} show how to deduce approximation results for piecewise linear functions on triangular meshes from piecewise bilinear functions on rectangles. Altogether, the proof of Theorem  \ref{starthm} is complete.
\end{proof}

\section{Algorithmic details} \label{algo}
The a posteriori error estimate from Theorem A leads to an adaptive mesh refinement procedure, based on the four steps:
 \begin{align*}
  \textbf{SOLVE}&\longrightarrow \textbf{ESTIMATE}\longrightarrow \textbf{MARK}\longrightarrow \textbf{REFINE}.
\end{align*}
The precise algorithm is given as follows:\\

\noindent \textbf{Adaptive Algorithm:}\\
Input: Spatial mesh $\mathcal{T}=\mathcal{T}_0$, refinement parameter $\theta  \in (0, 1)$, tolerance $\epsilon > 0$, data $f$.
\begin{enumerate}
\item Solve $\mathcal{V} \dot{\varphi}_{h,\Delta t}=\dot{f}$ on $\mathcal{T}$.
\item Compute the error indicators $\eta(\bigtriangleup)$ in each triangle $\bigtriangleup\in \mathcal{T}$.
\item Find $\eta_{max}=\max_{\bigtriangleup} \eta(\bigtriangleup)$.
\item Stop if $ \sum_i \eta^2(\vartriangle_i) < \epsilon^2$.
\item Mark all $\bigtriangleup\in \mathcal{T}$ with $\eta(\vartriangle_i)> \theta \eta_{max}$.
\item Refine each marked triangle into 4 new triangles to obtain a new mesh $\mathcal{T}$\\ Choose $\Delta t$ such that $\frac{\Delta t}{\Delta x} \leq 1$ for all triangles.
\item Go to 1.
\end{enumerate}
Output: Approximation of $\dot{\varphi}$.\\

In the first step, we solve $V \dot{\varphi} = \dot{f}$ using the Galerkin discretization \eqref{DPdisc} 
in $V^{1,1}_{h, \Delta t}$.
The Galerkin solution has the form
\begin{equation*}
\dot{\varphi}_{h,\Delta t}(x,t)=\sum_{m=1}^{N_t}\sum_{i=1}^{N_s} \varphi_i^m \beta^m(t)\xi_i(x)\ ,
\end{equation*}
where $\beta^m$ is the piecewise linear hat function in time associated to time $t_m$,
\begin{align*}
\beta^m(t)=(\Delta t)^{-1}((t-t_m)\chi_{[t_{m-1},t_{m}]}(t)-(t-t_{m+1})\chi_{[t_{m},t_{m+1}]}(t)),
\end{align*}
and $\xi_i$ is the piecewise linear hat function in space associated to node $i$.

As a step towards adaptive mesh refinements in space-time, we here focus on time-integrated error indicators as they are relevant for geometric singularities. As shown in \cite{graded}, for polyhedral meshes and screens time-independent graded meshes lead to quasi-optimal convergence rates in spite of singularities of the solutions.  The time-integrated error indicator for triangle $\bigtriangleup$ is computed as
$$\eta^2(\bigtriangleup) = \textstyle{\sum_{n}}\left\{ \eta_{\bigtriangleup,\nabla_{\Gamma}}(I_n)^2+ \eta_{\bigtriangleup,\partial_t}(I_n)^2\right\}\ .$$
Here, for every triangle $\bigtriangleup$ and every time
interval $I_{n}=[t_{n-1}, t_{n}]$ we define the partial error indicators
\begin{align*}
\eta_{\bigtriangleup,\nabla_{\Gamma}}(I_n)^2 &= h_\bigtriangleup \int_{t_{n-1}}^{t_n} \int_{\bigtriangleup} [\nabla_{\Gamma}(\dot{f}-\mathcal{V}\dot{\varphi}_{h,\Delta t})]^2 ds_x dt\,,\\
\eta_{\bigtriangleup,\partial_t}(I_n)^2 & = \Delta t \int_{t_{n-1}}^{t_n} \int_{\bigtriangleup} [\partial_t (\dot{f}-\mathcal{V}\dot{\varphi}_{h,\Delta t})]^2 ds_x dt\,.
\end{align*}
The time integral is approximated by the trapezoidal rule, and the tangential gradient of a function $F$ is computed as
\begin{equation*}
 \nabla_\Gamma F(t,x)=P_\Gamma \nabla F = \nabla F(t,x)-\nu ( \nu\cdot \nabla F(t,x))\ 
\end{equation*}
with the outer unit normal vector $\nu$ to $\Gamma$, resp.~the projection $P_\Gamma$ onto the tangent bundle of $\Gamma$.\\

To compute $\eta_{\bigtriangleup,\nabla_{\Gamma}}$ from $\dot{\varphi}_{h,\Delta t}$, we consider the gradient of $\mathcal{V} \dot{\varphi}_{h,\Delta t}$ as a singular integral:
\begin{align*}
& \nabla_\Gamma \mathcal{V}\dot{\varphi}_{h,\Delta t}(t,x)\\&= \frac{-1}{4\pi} P_\Gamma \int_\Gamma (x-y)\left(\frac{\dot{\varphi}_{h,\Delta t}(t-|x-y|,y)}{|x-y|^3}+\frac{\ddot{\varphi}_{h,\Delta t}(t-|x-y|,y)}{|x-y|^2} \right) ds_y\, \\
 &= \frac{-1}{4\pi}\sum_{m=1}^{N_t}\sum_{i=1}^{N_s} \varphi_i^m P_\Gamma \int_\Gamma \varphi_i(y)\Big[\beta^m(t-|x-y|)\frac{x-y}{|x-y|^3} \\
 & \qquad + \dot{\beta}^m(t-|x-y|)\frac{x-y}{|x-y|^2} \Big] ds_y .
\end{align*}
Using the explicit form of $\beta^m$, we obtain  
\begin{align*}
 \nabla_\Gamma \mathcal{V}\dot{\varphi}_{h,\Delta t}(t,x) &  = \frac{-1}{4 \pi} \sum_{m=1}^{N_t} \sum_{i=1}^{N_s}  \varphi_{i}^{m} \Bigg[ \frac{t-t_{m-1}}{\triangle t} P_\Gamma \int_{t-t_m \leq |x-y| \leq t-t_{m-1}} \xi_{i}(y)  \frac{x-y}{|x-y|^{3}} ds_y \\
   &\qquad - \frac{t-t_{m+1}}{\triangle t}  P_\Gamma \int_{t-t_{m+1} \leq |x-y| \leq t-t_{m}} \xi_{i}(y) \frac{x-y}{|x-y|^{3}} ds_y \Bigg]\ .
\end{align*}
The integrals are evaluated with a composite hp-graded quadrature, like the entries of the BEM Galerkin matrix in  \eqref{DPdisc}. See  \cite{gimperleinreview} for details.

\section{Towards space-time adaptivity}\label{towardsst}


\textcolor{black}{The previous sections presented an adaptive mesh refinement procedure based on time-independent meshes. While these provide efficient approximations for solutions with time-independent, geometric singularities, wave phenomena naturally include singularities which move in space-time, such as travelling wave crests. }

\textcolor{black}{The time-averaged error indicators used above are inadequate in such settings. However, the underlying a posteriori error estimates still apply and can be used to define residual error indicators $\eta$ in each space-time element. They lead to an adaptive algorithm, as before based on the steps
 \begin{align*}
  \textbf{SOLVE}&\longrightarrow \textbf{ESTIMATE}\longrightarrow \textbf{MARK}\longrightarrow \textbf{REFINE}.
\end{align*}
The precise algorithm with local time stepping is given as follows:}\\

\noindent \textcolor{black}{\textbf{Space--time Adaptive Algorithm:}\\
 Input: Mesh $\mathcal{T}=(\mathcal{T}_S \times \mathcal{T}_T)_0$, refinement parameter $\theta  \in (0, 1)$, tolerance $\epsilon > 0$, data $f$.
\begin{enumerate}
\item Solve $V \dot{\varphi}_{h,\Delta t}=\dot{f}$ on $\mathcal{T}$.
\item Compute the error indicators $\eta(\Box)$ in each space-time prism $\Box \in \mathcal{T}$.
\item Find $\eta_{max}=\max_{\Box} \eta(\Box)$.
\item Stop if $\sum_i \eta^2(\Box_i) < \epsilon^2$.
\item Mark all $\Box\in \mathcal{T}$ with $\eta(\Box) > \theta \eta_{max}$.
\item Refine each marked $\Box$ in space to obtain a new mesh $\mathcal{T}$. If $\frac{\Delta t}{\Delta x} \geq 1$ in a refined element, divide local time step $\Delta t$ by $2$.
\item Go to 1.
\end{enumerate}
}

\begin{figure}[h]
  \centering
   \includegraphics[width=4.8cm]{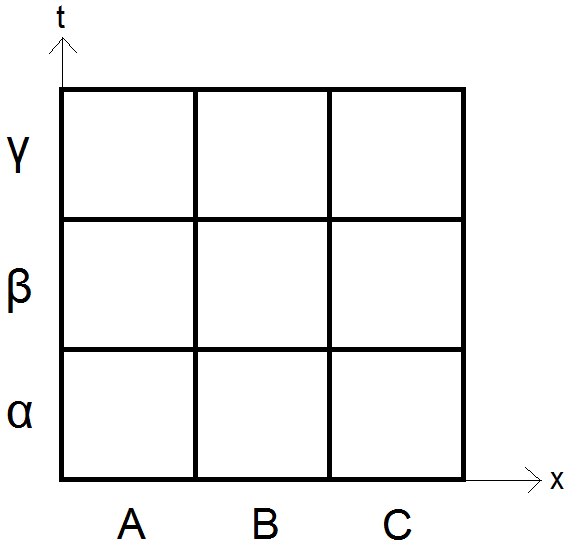}
   \includegraphics[width=5cm]{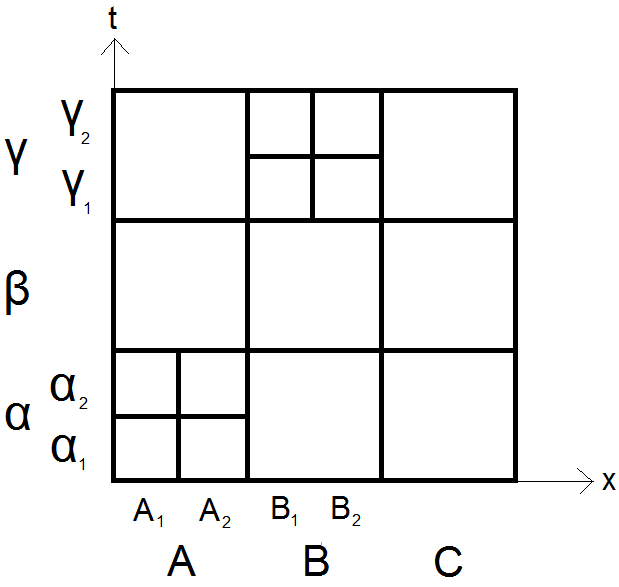}
   \caption{Uniform coarse mesh and local space-time refinement.}\label{spacetime}
\end{figure}

\textcolor{black}{Such fully space--time adaptive mesh refinements based on different error indicators have been explored by M.~Gl\"{a}fke \cite{Glaefke} for 2d problems, and  Figure \ref{spacetime} gives a schematic illustration of the space-time refinement used in his work.}
 
\textcolor{black}{The flexibility of the space-time adaptive approach comes with additional computational cost. When the space-time mesh is a global product of spatial and temporal meshes with equidistant time steps, the Galerkin matrix has a block Toeplitz structure with constant blocks $V^j$ corresponding to global time steps $t_n$, $t_m$ with $j=n-m$. }

\textcolor{black}{The Toeplitz structure is lost when the time step or spatial mesh change with time. Instead of one new matrix $V^j$ in time step $j$, for variable time step a naive implementation even for a constant spatial now requires the computation of $j+1$ matrices in time step $j$. While Gl\"{a}fke's brute force approach in 2d \cite{Glaefke} provides a proof of principle for the adaptive procedure,  it becomes computationally infeasible in 3d. Efficient implementations of space-time adaptivity for time domain boundary elements would likely reuse those entries of the space-time system which have not been affected by the last refinement step. The actual implementation remains a challenge for future work in both 2d and 3d.}

\section{Numerical experiments}\label{experiments} 

${}$\\
\noindent \textbf{Example 1:} We consider the Dirichlet problem $\mathcal{V} \phi = f$ on {the unit sphere} $\Gamma = S^2$ with the right hand side $f(t,x,y,z) = \sin(t)^5x^2$ and $[0, T] = [0, 2.5]$. We use a discretization by linear ansatz and test functions in space and time. $\Gamma$ is approximated by uniform meshes of 80, 320,1280, and 5120 triangles, and the time step $\Delta t$ is $0.4$, $0.2$, $0.1$, resp.~$0.05$ for the respective meshes to keep $\frac{\Delta t}{h}$ fixed. The numerical results are compared to the exact solution. 

Figure \ref{sphereerror} shows the convergence of the error $\phi - \phi_{h,\Delta t}$ in the energy norm as well as the $L^2$ error in the sound pressure and compares them to both residual and ZZ error indicators. The resulting convergence rates are similar: We obtain a convergence rate of 0.93 in energy norm, 0.97 in sound pressure, 0.9 in the residual error indicator, and 1.02 in the ZZ indicator.
This illustrates the reliability and efficiency of both error indicators with respect to the energy norm and related quantities such as the sound pressure in an example with known exact solution.  More precisely, the quotient of the error estimate and the energy error, the efficiency index,  remains approximately constant at $0.025$ as the number of degrees of freedom increases.\\

\begin{figure}{
\centering
 \makebox[\linewidth]{
 \includegraphics[width=9cm]{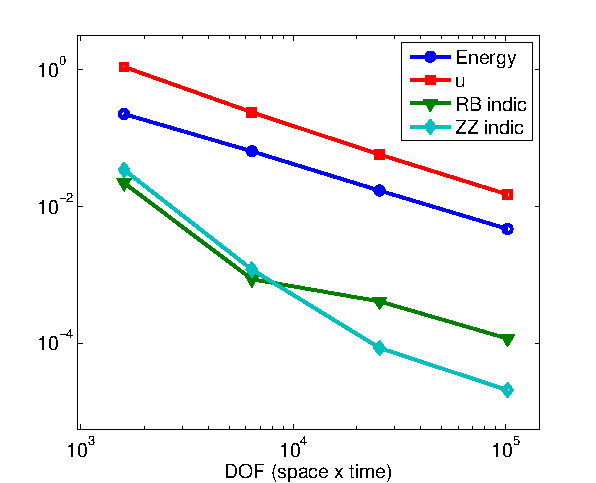}
 }
\caption{Energy error, residual and ZZ error indicators for Dirichlet problem on $\Gamma = S^2$, Example 1. }
\label{sphereerror} }
\end{figure}

\noindent \textbf{Example 2:} We consider the Dirichlet problem $\mathcal{V} \phi = f$ on the square screen $\Gamma = [-0.5, 0.5]^2\times \{0\}$ with the right hand side $f(t,x,y,z) = \sin(t)^5x^2$ for times $[0,2.5]$. Using a discretization by linear ansatz and test functions in space and time, we compare the error of a uniform discretization to the error of an adaptive series of meshes, steered by the residual error estimate.  The time step is fixed at $\Delta t = 0.1$, and the uniform meshes consist of $18$, $288$, $648$, $1352$, and $6050$ triangles, while the adaptive refinements correspond to $36$, $74$, $164$, $370$, $784$, $1676$, $3485$, and $7432$ triangles. \\

Figure \ref{convwop} shows the convergence of the error indicator and the error in the energy norm, for both the uniform and adaptive series of meshes. The convergence rate is approximately $0.48$ for uniform refinements, compared to $0.77$ for adaptive refinements. The convergence rate in the uniform case agrees with the theoretical prediction of $0.5$ from \cite{graded}, and the adaptive convergence rate of $0.77$ recovers the results for time-independent screen problems \cite{cms}. \\

As in the elliptic case, the convergence rate of the adaptive refinements does not reach the optimal rate of $1.5$ achieved with algebraically graded meshes, as demonstrated in \cite{graded}. The optimal anisotropic graded meshes cannot be obtained by mesh refinements: While adaptive meshes are locally quasi-uniform, graded meshes involve arbitrarily thin triangles with shallow angles near the edges of the screen. A heuristic explanation for  the  substantially higher rates of (anisotropic) graded meshes is contained in \cite{cmsp}.
\\

Figure \ref{screenmeshes} shows representative adaptive meshes, where the color scale highlights the residual-based indicator values for each element. Mesh refinements concentrate at the left and right edges, where the right hand side is steep, and to a lesser extent also at the top and bottom edges.\\

\begin{figure}
  \centering
   \includegraphics[width=9cm]{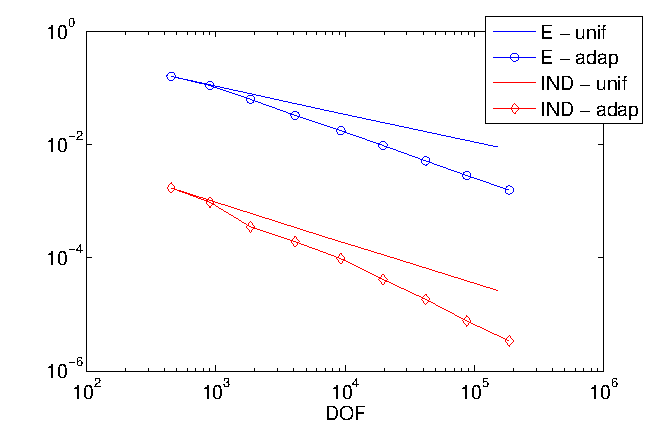}
   \caption{Energy error and residual error indicators for Dirichlet problem on $\Gamma = [-0.5, 0.5]^2\times \{0\}$, Example 2.}\label{convwop}
\end{figure}

\begin{figure}
  \centering
   \includegraphics[width=5cm]{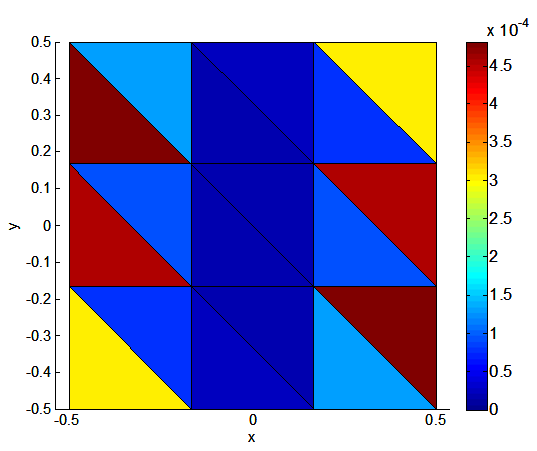}
   \includegraphics[width=5cm]{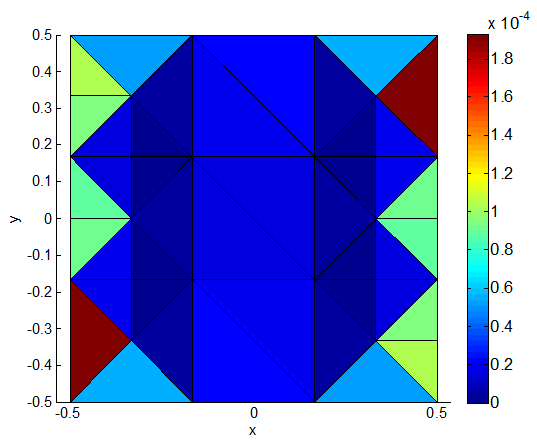}
   \includegraphics[width=5cm]{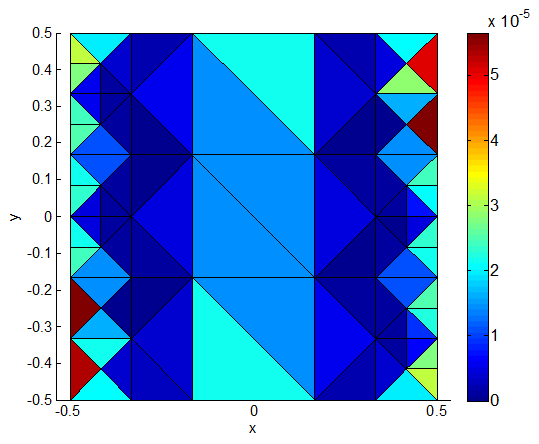}
   \includegraphics[width=5cm]{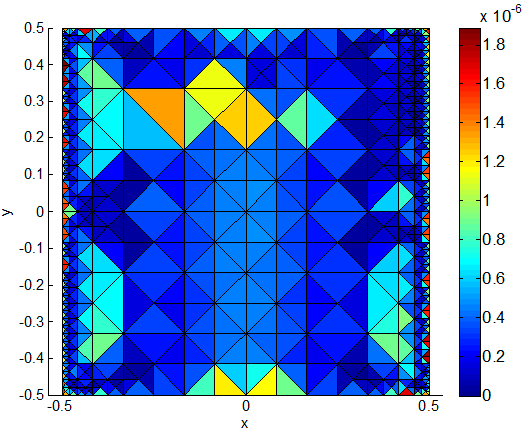}
   \caption{Meshes 1, 2, 3 and 6 generated by adaptive refinements, Example 2.}\label{screenmeshes}
\end{figure}

\noindent \textbf{Example 3:} We consider the Dirichlet problem $\mathcal{V} \phi = f$ on {the triangle $\Gamma$ with angles of $45$, $45$ and $90$ degrees, as  depicted in Figure \ref{isotrianglemesh}}. The right hand side is given by $f(t,x,y,z) = \sin(t)^5$, and we consider times $[0,2.5]$. Using the discretization from Example 2, we compare the error on uniform meshes to the error of an adaptive series of meshes, steered by the residual error estimate.  The time step is fixed at $\Delta t = 0.1$.\\ 

Figure \ref{isotriangle} shows the convergence of the error indicator and the error in the energy norm, for both the uniform and adaptive series of meshes. The convergence rate is approximately $0.49$ for uniform refinements, compared to $0.78$ for adaptive refinements, almost identical to the square screen in Example 2. \\

Figure \ref{isotrianglemesh} shows representative adaptive meshes, where the color scale highlights the residual-based indicator values for each element. As expected, mesh refinements concentrate in the two sharper corners of the triangle.\\

\begin{figure}
  \centering
   \includegraphics[width=9cm]{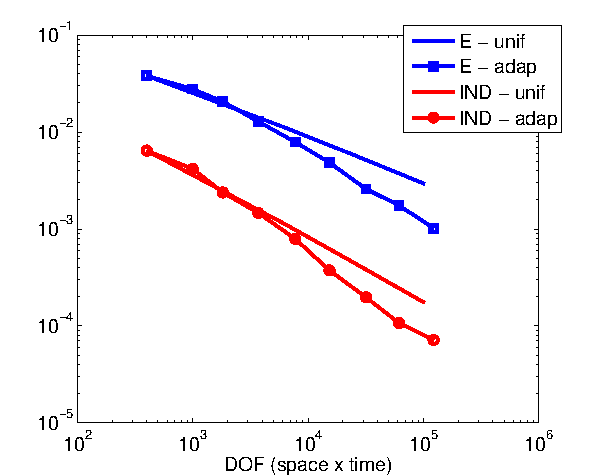}
   \caption{Energy error and residual error indicators for Dirichlet problem isosceles triangle, Example 3.}\label{isotriangle}
\end{figure}

\begin{figure}
  \centering
   \includegraphics[width=5cm]{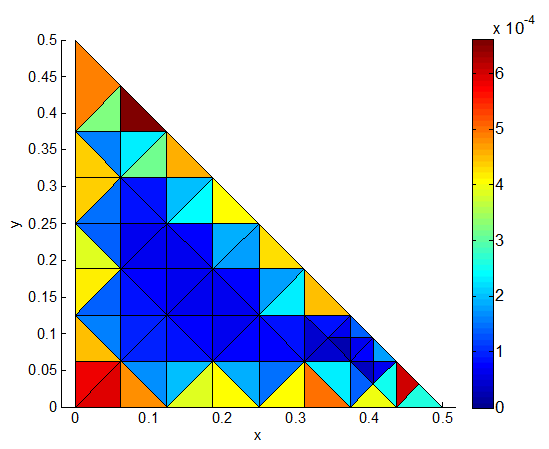}
   \includegraphics[width=5cm]{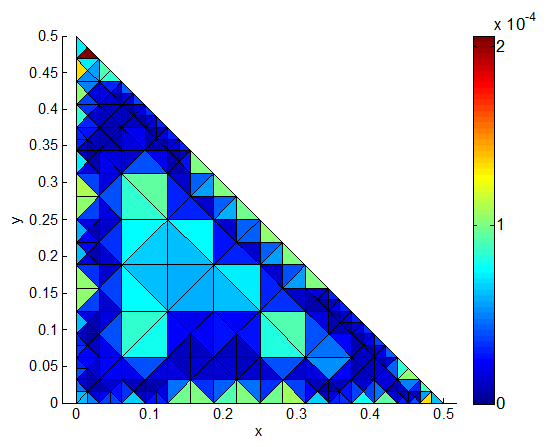}
   \includegraphics[width=5cm]{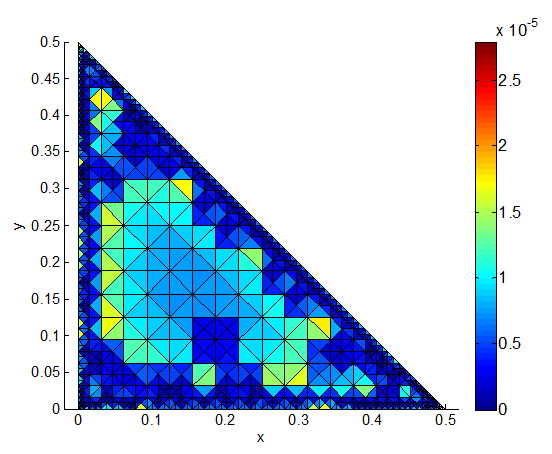}
   \includegraphics[width=5cm]{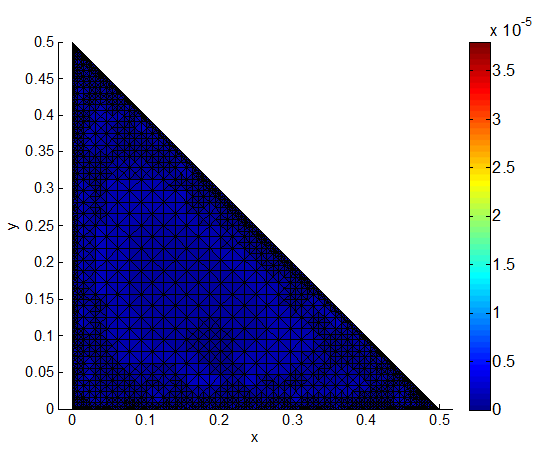}
   \caption{Meshes 3, 5, 7 and 8  generated by adaptive refinements, Example 3.}\label{isotrianglemesh}
\end{figure}

\noindent \textbf{Example 4:}  We consider the Dirichlet problem $\mathcal{V} \phi = f$ on the triangle $\Gamma$ with angles of $30$, $60$ and $90$ degrees, as depicted in Figure \ref{369mesh}. The right hand side is given by $f(t,x,y,z) = \sin(t)^5$, and we consider times $[0,2.5]$. Using the discretization from Example 2, we compare the error on uniform meshes to the error of an adaptive series of meshes, steered by the residual error estimate. The time step is fixed at $\Delta t = 0.1$.\\ 

Figure \ref{369triangle} shows the convergence of the error indicator and the error in the energy norm, for both the uniform and adaptive series of meshes. The convergence rate is approximately $0.448$ for uniform refinements, compared to $0.65$ for adaptive refinements. The rates are slightly reduced compared to Examples 2 and 3, possibly because the asymptotic regime only sets in for higher degrees of freedom because of the small angles of $30$ degrees in the triangulation.\\

Figure \ref{369mesh} shows representative adaptive meshes, where the color scale highlights the residual-based indicator values for each element. As expected, mesh refinements concentrate in the corners according to their sharpness.\\

\begin{figure}
  \centering
   \includegraphics[width=9cm]{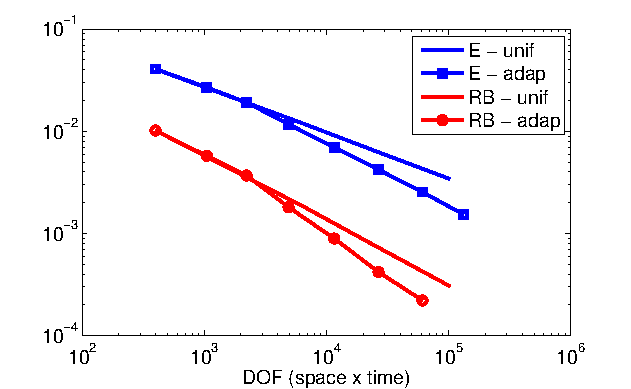}
   \caption{Energy error and residual error indicators for Dirichlet problem on 30-60-90 triangle, Example 4.}\label{369triangle}
\end{figure}

\begin{figure}
  \centering
   \includegraphics[width=5cm]{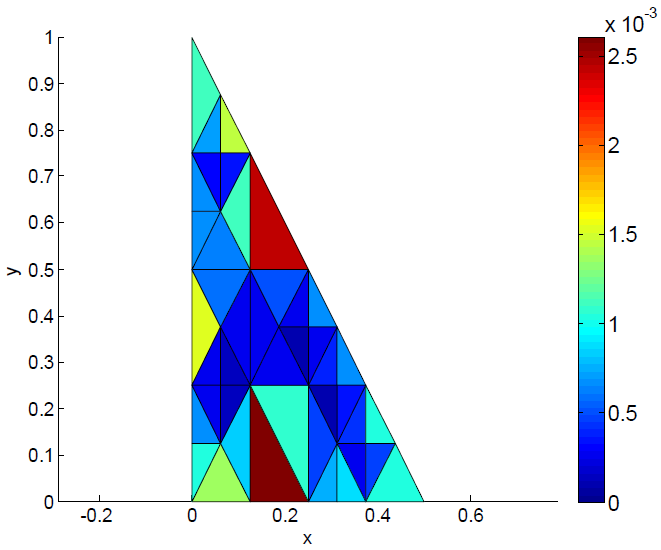}
   \includegraphics[width=5cm]{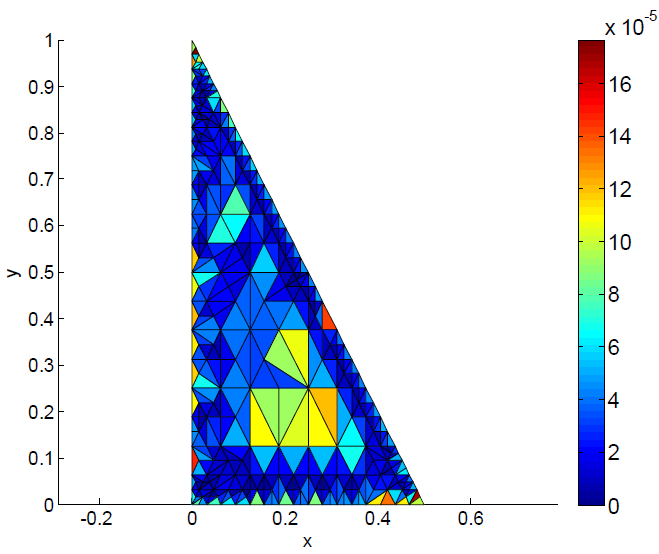}
   \includegraphics[width=5cm]{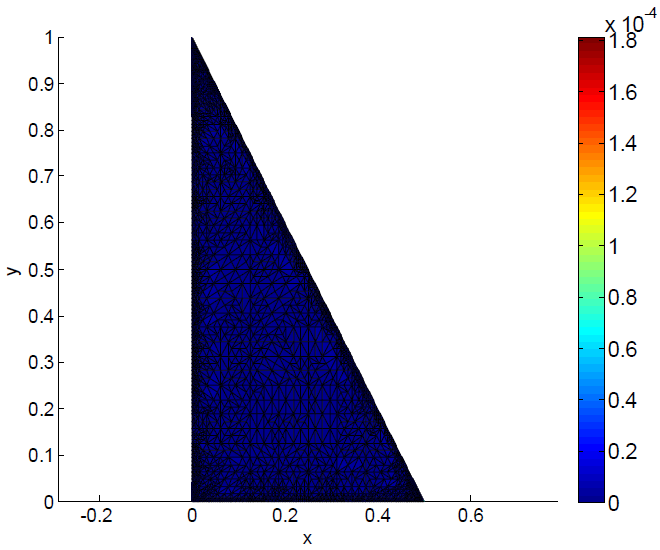}
   \caption{Meshes 2, 5, 8  generated by adaptive refinements, Example 4.}\label{369mesh}
\end{figure}

From Experiments 2, 3 and 4 we conclude that the convergence rate is does not depend on the angles of the triangle, and therefore the corner singularity. The convergence rate of around $\frac{1}{2}$ on uniform meshes  matches the rate theoretically expected for the approximation of the edge singularity \cite{graded}, while the approximation  error from the corner singularities is of higher order. The adaptive convergence rates of around $0.78$ are compatible with the convergence rates of around $0.8$ for the time-independent Laplace equation in \cite{cms}. The rates are slightly reduced in Example 4, with angles of $30$ degrees, possibly because of the necessarily thin triangles in the triangulation.

\section{Appendix: Mapping properties}

We consider the mapping properties of Theorem \ref{mapthm2}, for $\Gamma$ Lipschitz. The key step involves estimates for a fundamental solution to the wave equation, from which the mapping properties of the layer potentials can be deduced similarly as in Costabel's work for the Laplace equation \cite{cos}. \\

Recall the following result by Becache and Ha Duong \cite{b1} for the Dirichlet trace:
\begin{lemma}\label{tracelemma}
For $s\in(\frac12,\frac32),\gamma_0\colon H^{s}_{\omega,loc}(\mathbb{R}^d)\to H^{s-\frac12}_\omega(\Gamma)$ continuous and $\Vert \gamma_0\, u\Vert_{s-\frac12 ,\omega, \Gamma}\lesssim_{\sigma} \Vert u\Vert_{s,\omega, \mathbb{R}^d}$.
\end{lemma}
Becache and Ha Duong only state this result for $s\leq 1$. The extension to $s>1$ relies on the approach of \cite{cos} for Lipschitz $\Gamma$. For $\Gamma$ of class $C^{1,\alpha}$, continuity also holds for $s=\frac{3}{2}$, by the extension in \cite{wendland}.\\ 

For $x\in \mathbb{R}^{d}\setminus \Gamma$, the single layer operator for the Helmholtz equation is given by \begin{align}
     K_{0}^{\omega}v(x)&=\int_{\Gamma} G_{\omega}(x,y) v(y)\; ds_{y} =  G_{\omega}\circ \gamma_{0}^{*}v(x)\ ,
    \end{align}
where $\gamma_{0}^{*}$ is the adjoint of the trace map $\gamma_{0}$.

As before, we always consider frequencies $\omega \in \mathbb{C}$ with $\mathrm{Im} \omega > \sigma$. Note that
\begin{align}
\frac{1}{i \bar{\omega}} G_{\omega} : H^{-s}_{\omega, comp}(\mathbb{R}^{d})\rightarrow  H^{-s+2}_{\omega,loc}(\mathbb{R}^{d})\ ,
\end{align}
because
\begin{align*}
\|\frac{1}{i \overline{\omega}} G_\omega v \|_{-s+2, \omega, \mathbb{R}^{d}} & = \| \mathcal{F}^{-1}\frac{1}{i\overline{\omega}}(|\xi|^2-\omega^2)^{-1} \mathcal{F}v\|_{-s+2,\omega,\mathbb{R}^{d}}\\
& = \| (|\xi|^2 + |\omega|^2)^{\frac{-s+2}{2}} \frac{1}{i\overline{\omega}}(|\xi|^2-\omega^2)^{-1} \mathcal{F}v\|_{L^2(\mathbb{R}^{d})}\\
& \leq  (\mathrm{Im} \omega)^{-1}\ \| (|\xi|^2 + |\omega|^2)^{\frac{-s+2}{2}} (|\xi|^2+|\omega|^2)^{-1} \mathcal{F}v\|_{L^2(\mathbb{R}^{d})}\\
& = (\mathrm{Im} \omega)^{-1}\ \| v \|_{-s,\omega,\mathbb{R}^{d}}\ .
\end{align*}
Here
\begin{align*}
\mathrm{Re} \ i\overline{\omega} (|\xi|^2-\omega^2) = (\mathrm{Im} \ \omega) (|\xi|^2 + |\omega|^2)
\end{align*}
Hence
\begin{align}
 \Arrowvert G_{\omega} f \Arrowvert_{-s+2, \omega,\mathbb{R}^{d}}\lesssim \frac{|\omega|}{\mathrm{Im} \omega}\Arrowvert f \Arrowvert_{-s, \omega,\mathbb{R}^{d}},
\end{align}
so that with $\mathrm{Im} \omega\geq \sigma$
\begin{align}\label{singpotref}
 \Arrowvert  K_{0}^{\omega}v \Arrowvert_{-s+2,\omega,\mathbb{R}^{d}}\lesssim_\sigma |\omega|
  \Arrowvert  v \Arrowvert_{-s+\frac{1}{2}, \omega,\Gamma}\ .
\end{align}

From Lemma \ref{tracelemma} we conclude that $\mathcal{V}^{\omega}:=\gamma_{0} K_{0}^{\omega}:H^{-\frac{1}{2}+\tau}_\omega(\Gamma)\rightarrow H^{\frac{1}{2}+\tau}_\omega(\Gamma)$ continuously for $\tau\in (-\frac{1}{2},\frac{1}{2})$, and
\begin{align}
 \Arrowvert \mathcal{V}^{\omega}v\Arrowvert_{\frac{1}{2}+\tau, \omega, \Gamma}\lesssim_\sigma |\omega|  \Arrowvert v\Arrowvert_{-\frac{1}{2}+\tau, \omega,\Gamma}\ .
\end{align}

For  $x\in \mathbb{R}^{d}\setminus \Gamma$, the double layer potential for the Helmholtz equation is given by \begin{align}  K_{1}^{\omega}v(x)&=\int_{\Gamma}\partial_{\nu(y)} G_{\omega}(x,y) v(y)\; ds_y.
    \end{align}
To describe the mapping properties of this operator, we rely on the following lemma:
\begin{lemma}The Dirichlet problem
 \begin{subequations} \label{prob1}
\begin{alignat}{2}
Pu:=\Delta u+\omega^{2} u&=0 \ , \\
u|_{\Gamma}&=v,
\end{alignat}
\end{subequations}
for given $v\in H^{\frac{1}{2}}_\omega(\Gamma)$, admits a unique weak solution $u = Tv$, and
\begin{align}
\Arrowvert u \Arrowvert_{1,\omega,\Omega}\lesssim C(\sigma) |\omega|\Arrowvert v\Arrowvert_{\frac{1}{2},\omega,\Gamma}\ 
\end{align}
and
\begin{align}
  \Arrowvert T v \Arrowvert_{H^{1}_{P}(\Omega)}&:=\left( \Arrowvert u \Arrowvert_{1,\omega,\Omega}^{2}+\Arrowvert P u \Arrowvert^{2}_{0,\omega, \Omega}\right)^{\frac{1}{2}}
  \leq C'(\sigma) |\omega|\Arrowvert v\Arrowvert_{\frac{1}{2},\omega,\Gamma} .\label{H1P}
\end{align}
\end{lemma}
\begin{proof}
 The bilinear form
 \begin{align}
  a(u,u)=-(\nabla u, \nabla u)_{L^{2}(\Omega)}+\omega^{2}(u,u)_{L^{2}(\Omega)}
 \end{align}
satisfies
\begin{align}
 \text{Re}\left\{ \frac{(-i \bar{\omega})}{|\bar{\omega}|}a(u,u)\right\}=-\frac{\text{Im}\omega}{|\bar{\omega}|}\Arrowvert u \Arrowvert_{1,\omega,\Omega}^{2}\ .
\end{align}
Hence, the associated operator $-i\bar{\omega} A$ satisfies
\begin{align}\label{eq1}
 \Arrowvert A^{-1}\Arrowvert_{\mathcal{L}(H^{-1}_\omega(\Omega),H^{1}_\omega(\Omega))}\leq \frac{|\omega|}{\text{Im}\omega}\ .
 \end{align}
We use the extension operator to extend $v$ to $\tilde{v}\in H^{1}_\omega(\Omega)$ with norm
\begin{align}
 \Arrowvert \tilde{v}\Arrowvert_{1,\omega,\Omega}\lesssim_\sigma \Arrowvert v\Arrowvert_{\frac{1}{2},\omega,\Gamma}.
\end{align}
Then we seek a solution in $H_{\omega, 0}^{1}(\Omega)$ to $A u=-A \tilde{v} \in H^{-1}_\omega(\Omega)$. By \eqref{eq1}, u exists and
\begin{align}
\Arrowvert u \Arrowvert_{1,\omega,\Omega}\lesssim_\sigma \frac{|\omega|}{\text{Im}\omega}\Arrowvert v\Arrowvert_{\frac{1}{2},\omega,\Gamma}.
\end{align}
\end{proof}

To relate the double layer potential $K_{1}^{\omega}$ to the solution operator $T$ of \eqref{prob1}, we use the representation formula for $x\in  \mathbb{R}^{d}\setminus \Gamma$:
\begin{align}
 u(x)=G_{\omega} f(x)+\langle \gamma_{1} G_{\omega}(x,\cdot),[\gamma_{0} u]\rangle-\langle[\gamma_{1} u],G_{\omega}(x,\cdot)\rangle \ .
\end{align}
Here $\gamma_1$ denotes the Neumann trace. This shows
\begin{align}
 T v=- K_{1}^{\omega}v+ K_{0}^{\omega}\gamma_{1} T v,
\end{align}
or
\begin{align}\label{K1formula}
 K_{1}^{\omega}=(-1+ K_{0}^{\omega}\gamma_{1}) T\ .
\end{align}
With the operator norms from \eqref{singpotref} and \eqref{H1P}, we conclude 
\begin{align}
& \Arrowvert K_{1}^\omega v\Arrowvert_{1,\omega,\Omega}\nonumber \\&\lesssim \left( 1+\Arrowvert K_{0}^{\omega}\Arrowvert_{\mathcal{L}(H^{-\frac{1}{2}}_\omega(\Gamma), H^{1}_\omega(\Omega))}
 \Arrowvert \gamma_{1}\Arrowvert_{\mathcal{L}(H^{1}_{P}(\Omega), H^{-\frac{1}{2}}_\omega(\Gamma))}\right)\Arrowvert T\Arrowvert_{\mathcal{L}(H^{\frac{1}{2}}_\omega(\Gamma), H^{1}_{P}(\Omega))}
 \Arrowvert v\Arrowvert_{\frac{1}{2},\omega,\Gamma}\nonumber\\
 &\lesssim_\sigma \left(1+|\omega|
 \Arrowvert \gamma_{1}\Arrowvert_{\mathcal{L}(H^{1}_{P}(\Omega), H^{-\frac{1}{2}}_\omega(\Gamma))}\right)
 C'(\sigma)|\omega|\Arrowvert v\Arrowvert_{\frac{1}{2},\omega,\Gamma}\ .
\end{align}
It remains to determine $ \Arrowvert \gamma_{1}\Arrowvert_{\mathcal{L}(H^{1}_{P}(\Omega), H^{-\frac{1}{2}}_\omega(\Gamma))}$, which we now pursue.

The trace map $\gamma_{0}$ admits a right-inverse $\gamma_{0}^{-}$, which maps $H^{s-1/2}_\omega(\Gamma) \to H^{s}_{\omega, loc}(\mathbb{R}^d)$ continuously for all $s \in (1/2, 1]$.\\

With
\begin{align}
 \phi\mapsto\langle \gamma_{1} u, \phi \rangle:= -a(u,\gamma_{0}^{-}\phi)-\int_{\Omega} P u\; \gamma_{0}^{-}\bar{\phi}\ ,
\end{align}
we have
\begin{align}
& \sup_{\phi\neq 0}\frac{|\langle \gamma_{1} u, \phi \rangle|}{\Arrowvert \phi\Arrowvert_{\frac{1}{2},\omega,\Gamma}}\nonumber\\&=\sup_{\phi\neq 0}\frac{|-a(u,\gamma_{0}^{-}\phi)-\int_{\Omega} P u\; \gamma_{0}^{-}\bar{\phi}|}{\Arrowvert \phi\Arrowvert_{\frac{1}{2},\omega,\Gamma}}\nonumber\\
 &\leq \sup_{\phi\neq 0} \frac{|-\int_{\Omega}\partial_{\nu} u \overline{\partial_{j}(\gamma_{0}^{-}\phi})+\omega^{2} u \gamma_{0}^{-}\bar{\phi}-\int_{\Omega} P u\; \gamma_{0}^{-}\bar{\phi}|}{\Arrowvert \phi\Arrowvert_{H^{\frac{1}{2},\omega}(\Gamma)}}\nonumber\\
 &\leq \sup_{\phi\neq 0} \frac{1}{\Arrowvert \phi\Arrowvert_{\frac{1}{2},\omega,\Gamma}}\left(\Arrowvert \nabla u \Arrowvert_{L^{2}(\Omega)}\Arrowvert  \nabla \gamma_{0}^{-}\phi \Arrowvert_{L^{2}(\Omega)}
 +|\omega|^{2}\Arrowvert  u \Arrowvert_{L^{2}(\Omega)}\Arrowvert \gamma_{0}^{-}\phi \Arrowvert_{L^{2}(\Omega)}+\Arrowvert P u\Arrowvert_{L^{2}(\Omega)}\Arrowvert \gamma_{0}^{-}\phi \Arrowvert_{L^{2}(\Omega)}\right)\nonumber\\
  &\lesssim_{\sigma}\sup_{\phi\neq 0} \frac{1}{\Arrowvert  \gamma_{0}^{-}\phi\Arrowvert_{1,\omega,\Omega}}\left[\left(\Arrowvert \nabla u \Arrowvert^{2}_{L^{2}(\Omega)}
 +|\omega|^{2}\Arrowvert  u \Arrowvert^{2}_{L^{2}(\Omega)}\right)^{\frac{1}{2}}\left(\Arrowvert  \nabla \gamma_{0}^{-}\phi \Arrowvert^{2}_{L^{2}(\Omega)}+|\omega|^{2}\Arrowvert \gamma_{0}^{-}\phi \Arrowvert^{2}_{L^{2}(\Omega)}\right)^{\frac{1}{2}}\right.\nonumber\\
 &\left.\qquad +\Arrowvert P u\Arrowvert_{L^{2}(\Omega)}\Arrowvert \gamma_{0}^{-}\phi \Arrowvert_{L^{2}(\Omega)}\right]\nonumber\\
 &\lesssim_\sigma \left(\Arrowvert \nabla u \Arrowvert^{2}_{L^{2}(\Omega)}
 +|\omega|^{2}\Arrowvert  u \Arrowvert^{2}_{L^{2}(\Omega)}\right)^{\frac{1}{2}}+\Arrowvert P u\Arrowvert_{L^{2}(\Omega)}
\end{align}
Therefore we have for the conormal derivative $\gamma_1$:
\begin{lemma}
 Let $u\in H_{P}^{1}(\Omega)$. Then $\varphi\mapsto \langle\gamma_{1} u, \varphi \rangle$ is a continuous linear functional on $H^{\frac{1}{2}}_\omega(\Gamma)$ and 
 \begin{align}
  \Arrowvert\gamma_{1} u \Arrowvert_{-\frac{1}{2},\omega,\Gamma}\lesssim_{\sigma} \Arrowvert u\Arrowvert_{H^{1}_{P}(\Omega)}.
 \end{align}
\end{lemma}
We conclude for the double layer potential in the energy space:
\begin{equation}\label{K1estimate}
 \Arrowvert K_{1}^\omega v\Arrowvert_{1,\omega,\Omega} \lesssim_{\sigma}  |\omega|\left(1+|\omega|\right)\Arrowvert v\Arrowvert_{\frac{1}{2},\omega,\Gamma}.
\end{equation}
Variational arguments show that $\Arrowvert K_{1}^\omega v\Arrowvert_{1,\omega,\Omega}\lesssim   \frac{|\omega|}{\text{Im}\omega}\Arrowvert v\Arrowvert_{\frac{1}{2},\omega,\Gamma}$. However, the above argument generalizes \eqref{K1estimate} to arbitrary Sobolev exponents.\\

This generalization relies on the following theorem, which specifies the $\omega$-dependence of the endpoint estimates for the Dirichlet-Neumann  and Neumann-Dirichlet operators, denoted by $\gamma_1T$, respectively $ND$ \cite{necas}:
\begin{theorem}\label{necastheorem} For all $\tau \in [-1/2, 1/2]$:\\
a) $
 \lVert \gamma_1T \rVert_{\mathcal{L}(H^{\tau+1/2}_\omega(\Gamma),  H^{\tau-1/2}_\omega(\Gamma))} \lesssim_{\sigma} \lvert \omega \rvert$\ ,\\
b) $\lVert N D \rVert_{\mathcal{L}(H^{\tau-1/2}_\omega(\Gamma), H^{\tau+1/2}_\omega(\Gamma))} \lesssim_{\sigma} \lvert \omega \rvert$\ .
\end{theorem}
Theorem \ref{necastheorem} will be used to prove the following lemma:
\begin{lemma}\label{LabelT}
 For $\tau\in [-\frac{1}{2},\frac{1}{2}]$, $T:H^{\frac{1}{2}+\tau}_\omega(\Gamma)\rightarrow H_{P}^{1+\tau}(\Omega)$ continuous and
 \begin{align}
  \Arrowvert T v\Arrowvert_{H_{P}^{1+\tau}(\Omega)}\lesssim_{\sigma} |\omega|\Arrowvert v\Arrowvert_{\frac{1}{2}+\tau, \omega,\Gamma}.
 \end{align}
\end{lemma}
\begin{lemma}
 For $s \in (\frac{1}{2},\frac{3}{2})$, $\gamma_{1}:H^{s}_{P}(\Omega)\rightarrow H^{s-\frac{3}{2}}_\omega(\Gamma)$ continuous and $$\|\gamma_1 u\|_{s-\frac{3}{2},\omega, \Gamma} \lesssim_\sigma \|u\|_{H^s_P(\Omega)} \ .$$
\end{lemma}
\begin{proof}
This follows from the Costabel's trace theorem for $\omega = 1$, $\|\gamma_1 u\|_{H^{s-\frac{3}{2}}(\Gamma)} \lesssim \|u\|_{H^s(\Omega)}$, using that $s-\frac{3}{2}<0$:
$$\|\gamma_1 u\|_{s-\frac{3}{2},\omega, \Gamma} \lesssim_\sigma\|\gamma_1 u\|_{H^{s-\frac{3}{2}}(\Gamma)} \lesssim \|u\|_{H^s(\Omega)} \lesssim_\sigma \|u\|_{H^s_P(\Omega)} \ .$$
\end{proof}
\begin{proof}[Proof of Lemma \ref{LabelT}] Following \cite{cos}, for a large enough ball $B \supseteq \overline{\Omega}$ and $ \Omega_{2} = B \backslash \overline{\Omega}$, we consider 
\begin{subequations} \label{helmholtz}
\begin{alignat}{2}
\Delta u+\omega^{2} u&=0 \qquad \text{ in}\ \Omega_2,\\
u|_{\Gamma}&=v, \\
u|_{\partial B} &= 0 
\end{alignat}
\end{subequations}
with $v\in H^{\frac{1}{2}+\tau}_\omega(\Gamma)$. The solution operator is denoted by $T_2$: $u = T_2v$. \\ \ \\
Let $ u =\begin{cases} Tv &  \text{in}\ \Omega \\
                       T_{2} v & \text{in}\ \Omega_{2}
         \end{cases}.$
Then we have \\
\begin{align}\label{represent}  u=-K_{0}^{\omega}\gamma_{1}u + \int_{\partial B}\partial_{\nu}u(y)\ G_{\omega}(\cdot,y) \; ds_y,\quad \text{in}\ \Omega \cup \Omega_{2}
    \end{align}
By Theorem \ref{necastheorem} we have for $\tau \in [-1/2,1/2]$
\begin{align*}
 \lVert \partial_{\nu} u|_{\partial B}  \rVert_{-1/2+\tau,\omega,\partial B} + \lVert \gamma_{1} T v  \rVert_{-1/2+\tau,\omega,\Gamma} + \lVert \gamma_{1} T_{2} v  \rVert_{-1/2+\tau,\omega,\Gamma} \lesssim |\omega| \lVert v  \rVert_{1/2+\tau,\omega,\Gamma}
\end{align*}
Therefore with (\ref{represent}) we have
\begin{align*}
  \lVert u \rVert_{1+\tau,\omega,\Omega} \lesssim |\omega| \lVert v  \rVert_{1/2+\tau,\omega,\Gamma}
\end{align*}
yielding the assertion.
\end{proof}


%
%
%

We now prove the estimates in Theorem \ref{necastheorem}. For these, we rely on frequency-explicit Rellich identities, which we then translate into the time--domain.

\begin{proof}[Proof of Theorem \ref{necastheorem}] Applying the identity (5.1.1) and the Green's formula (5.1.2) in Ne\u{c}as \cite{necas}, yields with $Av = \sum_{j=1}^{N} \partial_{j}^{2} u  + \omega^{2}u$ and $ n_{k} h_{k} \geq C > 0 $ that
\begin{align} \label{rellich1}
& \int_{\partial \Omega} \left(- h_{k} n_{k} (\partial_{i} u )^{2} + 2 ( h_{i} \partial_{i} u ) ( n_{k} \partial_{k} u) \right) \nonumber \\&=
 \int_{\Omega} \Big( - (\partial_{k} h_{k}) (\partial_{i} u )^{2} + 2(\partial_{k} h_{i} ) (\partial_{i} u ) (\partial_{k} u ) - 2 h_{i} \omega^{2} (\partial_{i} u )
    u \Big) \ .
\end{align}
Here $n_k$ is the $k$-th component of the unit normal vector to $\Gamma$ and $h$ is a suitably chosen vector field.

Note that the left hand side is
$$
    2 \int_{\partial \Omega} \left( (h_{i} \partial_{i} u) (n_{k} \partial_{k} u ) - n_{k} h_{k} (\partial_{i} u)^{2}  \right)
      + \int_{\partial \Omega} h_{k} n_{k}(\partial_{i} u)^{2}\ .$$
Since the first integral only contains tangential derivatives,
$$
 \sum_{i} \lVert \partial_{i} u \rVert_{L^{2} ( \partial \Omega) }^{2} \lesssim  \lVert u |_{ \partial \Omega} \rVert_{H^{1}(\partial \Omega)}^{2} +
 \int_{\Omega} \lvert \triangledown u \rvert^{2} + \lvert \omega \rvert^{2} \Big| \int_{\Omega} (h_{i} \partial_{i} u ) u \Big|\ .
$$
Hence
\begin{equation}\label{onestar}
 \sum_{i} \lVert \partial_{i} u \rVert_{L^{2} ( \Gamma) }^{2} \lesssim  \lVert \triangledown_{\Gamma} u \rVert_{L^{2}(\Gamma)}^{2} +
 \int_{\Omega} \lvert \triangledown u \rvert^{2} + \lvert \omega \rvert^{2}  \int_{\Omega} \lvert \partial u \rvert  \lvert u \rvert 
\end{equation}
Next we consider
\begin{equation}\label{twostar}
 \int_{\Omega} \lvert \triangledown u \rvert^{2} - {\omega}^2 \lvert u  \rvert^{2} = \int_{\Gamma}  u|_{\partial \Omega} (\partial_{\nu} \overline{u} )\ .
\end{equation}
Taking the real part of \eqref{twostar} leads to
$$
\int_{\Omega} \lvert \triangledown u \rvert^{2} + \left( \lvert Im \omega \rvert^{2} - \lvert Re \omega \rvert^{2} \right) \lvert u \rvert^{2}
         = Re \int_{\Gamma} u ( \partial_{\nu} \overline{u} )\ ,
$$
while the imaginary part is given by
$$
2 (Im \omega) (Re \omega) \int_{\Omega} \lvert u \rvert^{2} = Im \int_{\Gamma} u ( \partial_{\nu} \overline{u} )\ .
$$
We consider two cases: First, for $ \lvert Re \omega \rvert \geq \frac{Im \omega}{2} \geq \frac{\sigma}{2} $:
\begin{align*}
  \int_{\Omega} \lvert u \rvert^{2} \lesssim \frac{1}{\lvert \omega \rvert Im \omega} \Big| \int_{\Gamma} u ( \partial_{\nu} \overline{u} ) \Big| \ ,\\
    \int_{\Omega} \lvert \triangledown u \rvert^{2} \overset{\eqref{twostar}}{\lesssim}
  \frac{(1+\lvert \omega \rvert)}{Im \omega} \Big| \int_{\Gamma} u ( \partial_{\nu} \overline{u} ) \Big| \ ,\\
   \int_{\Omega} \lvert \triangledown u \rvert^{2} + \lvert \omega \rvert^{2} \lvert u \rvert^{2} \lesssim
   \frac{(1+\lvert \omega \rvert)}{Im \omega} \Big| \int_{\Gamma} u ( \partial_{\nu} \overline{u} ) \Big|\ .
\end{align*}

In the remaining case, $ \lvert Re \omega \rvert \leq \frac{Im \omega}{2} ( \geq \frac{\sigma}{2} ) $,  we have
$ Im \omega \simeq \left( \lvert Im \omega \rvert^{2} - \lvert Re \omega \rvert^{2} \right)^{1/2}
                          \simeq \left( \lvert Im \omega \rvert^{2} + \lvert Re \omega \rvert^{2} \right)^{1/2} \simeq \lvert \omega \rvert$, and with \eqref{twostar}
 \begin{align*}\int_{\Omega} \lvert \triangledown u \rvert^{2} + \lvert \omega \rvert^{2} \lvert u \rvert^{2} \lesssim 
 \Big| \int_{\Gamma} u ( \partial_{\nu} \overline{u} ) \Big| \ ,\\
\int_\Omega\lvert \omega \rvert^{2} \lvert u \rvert \lvert \triangledown u \rvert  \lesssim \lvert \omega \rvert^{2} \lVert u \rVert_{L^{2}(\Omega)} \lVert \triangledown u \rVert_{L^{2}(\Omega)}\ , \\
\lVert \triangledown u \rVert_{L^{2}(\Omega)} \lesssim_{\sigma}(1+\lvert \omega \rvert)^{1/2} \Big| \int_{\Gamma} u ( \partial_{\nu} \overline{u} ) \Big|^{1/2} \ ,\\
\lVert u \rVert_{L^{2}(\Omega)}\lesssim \frac{1}{\lvert \omega
\rvert^{1/2}} \Big| \int_{\Gamma} u ( \partial_{\nu} \overline{u} ) \Big|^{1/2} \ .
 \end{align*}
This implies
 $$\int_{\Omega} \lvert \triangledown u \rvert^{2} + \lvert \omega \rvert^{2} \lvert u \rvert \lvert \triangledown u \rvert \lesssim_{\sigma}
   \lvert \omega \rvert^{2} \Big| \int_{\Gamma} u ( \partial_{\nu} \overline{u} ) \Big|\ . $$
Therefore \eqref{onestar} implies
$$
\sum_{i} \lVert \partial_{i} u \rVert_{L^{2}(\Gamma)}^{2} \lesssim_{\sigma} \lVert \triangledown_{\Gamma} u \rVert_{L^{2}(\Gamma)}^2 +
  \lvert \omega \rvert^{2} \lVert u \rVert_{L^{2}(\Gamma)} \lVert \partial_{\nu} u \rVert_{L^{2}(\Gamma)}\ ,$$ 

$$\sum_{i} \lVert \partial_{i} u \rVert_{L^{2}(\Gamma)} \lesssim_{\sigma} \lVert \triangledown_{\Gamma} u \rVert_{L^{2}(\Gamma)}+ \lvert \omega \rvert^{2} \lVert u \rVert_{L^{2}(\Gamma)}
   \simeq  |\omega| \lVert u \rVert_{1, \omega, \Gamma}\ ,
$$
i.e. Dirichlet data in $H^{1}_\omega(\Gamma)$ are mapped continuously to  Neumann data in $L^{2}(\Gamma)$.

\begin{equation*}
  \lVert \gamma_{1} T \rVert_{\mathcal{L}(H^{1}_\omega(\Gamma), L^{2}(\Gamma))} \lesssim_{\sigma} |\omega|\ .
\end{equation*}
Standard arguments using the divergence theorem now show:
\begin{align*}
  \int_{\Omega} div(\lvert u \rvert^{2} h ) &= \int_{\Omega} \lvert u \rvert^{2} div h + 2 \text{Re} \int_{\Omega} \overline{u} (h \cdot \triangledown u ) \\
  & \lesssim \int_{\Omega} \lvert u \rvert^{2} 
  +  \int_{\Omega} |\overline{u}| |\triangledown u| \\
& \lesssim_\sigma \frac{1}{ \lvert \omega \rvert} \Big| \int_{\Gamma} \overline{u} \partial_{\nu} u \Big| +  \lVert u \rVert_{L^{2}(\Omega)} \lVert \triangledown u \rVert_{L^{2}(\Omega)}\\
  & \lesssim_{\sigma} \left( \frac{1}{\lvert \omega \rvert}  + 1 \right) \Big| \int_{\Gamma} \overline{u} \partial_{\nu} u \Big| \\
  & \lesssim \left(\frac{1}{\lvert \omega \rvert} + 1 \right) \left( \epsilon \lVert u \rVert_{L^{2}(\Gamma)}^{2} + \frac{1}{\epsilon} \lVert \partial_{\nu} u \rVert_{L^{2}(\Gamma)}^2 \right)^{1/2},
\end{align*}
where $ 0 < \epsilon \lesssim \frac{\lvert \omega \rvert}{\rvert \omega \lvert +1}$. Therefore,
$$ \lVert u \rVert_{L^{2}(\Gamma)} \leq \left( \frac{1}{\lvert \omega \rvert} +1 \right)^{2} \lVert \partial_{\nu} u \rVert_{L^{2}(\Gamma)}.
$$
Using \eqref{rellich1} as in the proof of Lemma 5.2.2 in \cite{necas}, we conclude:
\begin{align*}
\int_{\Gamma} \left( -h_{k} n_{k} (\partial_{i} u)^{2} + 2 (h_{i} \partial_{i} u) (n_{k} \partial_{k} u) \right) \leq
\int_{\Omega}\left( \lvert \triangledown u \rvert^{2} + \lvert \omega \rvert^{2} \lvert \triangledown u \rvert \lvert u \rvert\right)\ .
\end{align*}
Note that the left hand side is larger than $\int_{\Gamma} \sum_{i} (\partial_{i} u)^{2} - C \sum_{i} \lvert \partial_{i} u \rvert \lvert \partial_{\nu} u \rvert$, so that
\begin{align*}
 \int_{\Gamma} \sum_{i} \lvert \partial_{i} u \rvert^{2} \lesssim \int_{\Gamma} (\partial_{\nu} u)^{2} + \int_{\Omega} \left( \lvert \triangledown u \rvert^{2} + \rvert \omega \lvert^{2} \lvert \triangledown u \rvert \lvert u \rvert \right)\ .
\end{align*}
As above,
\begin{align*}
  \int_{\Omega} \left( \lvert \triangledown u \rvert^{2} + \lvert \omega \rvert^{2} \lvert \triangledown u \rvert \lvert u \rvert \right) \lesssim_{\sigma} \lvert \omega \rvert^{2} \Big| \int_{\Gamma} u \partial_{\nu} \overline{u} \Big|\ ,
\end{align*}
so that 
\begin{align*}
 \sum_{i} \int_{\Gamma} \lvert  \partial_{i} u \rvert^{2} &\lesssim \int_{\Gamma} (\partial_{\nu} u)^{2} + \lvert \omega \rvert^{4} \lVert u \rVert_{L^{2}(\Gamma)}^2 \\
 & \lesssim \lvert \omega \rvert^{2} \lVert u \rVert_{1, \omega,\Gamma}^{2}\ .
\end{align*}
Altogether, we conclude the endpoint estimate $\|\gamma_1 T\|_{\mathcal{L}(H^{1}_\omega(\Gamma), L^{2}(\Gamma))} \lesssim_{\sigma} \lvert \omega \rvert$. \\

Further, as in \cite{necas}, Theorem 5.1.3, $ \gamma_1 T $ extends by duality to a bounded linear operator from $ L^{2}(\Gamma)$ to $H^{-1}_\omega(\Gamma) $ and
$$
 \lVert \gamma_1T \rVert_{\mathcal{L}(L^{2}(\Gamma), H^{-1}_\omega(\Gamma))} \lesssim_{\sigma} \lvert \omega \rvert\ .
$$
By interpolation, we conclude for $\tau \in [-1/2,1/2]$
$$
 \lVert \gamma_1T \rVert_{\mathcal{L}(H^{\tau+1/2}_\omega(\Gamma), H^{\tau-1/2}_\omega(\Gamma)} \lesssim_{\sigma} \lvert \omega \rvert\ .
$$
Similar arguments apply to the Neumann-Dirichlet operator $ND$. They lead to $$\lVert ND \rVert_{\mathcal{L}(H^{-1}_\omega(\Gamma), L^{2}(\Gamma))} \lesssim_{\sigma} \lvert \omega \rvert\ ,$$
and then by duality and interpolation for $\tau \in [-\frac{1}{2}, \frac{1}{2}]$
$$
 \lVert N D \rVert_{\mathcal{L}(H^{\tau-1/2}_\omega(\Gamma), H^{\tau+1/2}_\omega(\Gamma))} \lesssim_{\sigma} \lvert \omega \rvert\  .
$$
\end{proof}

We finally prove Theorem \ref{mapthm2}. 

\begin{proof}[Proof of Theorem \ref{mapthm2}] From above we recall for $\tau \in [-\frac{1}{2}, \frac{1}{2}]$
\begin{align}  
\lVert K_{0}^{\omega}v \rVert_{1+\tau,\omega,\Omega} \lesssim_\sigma |\omega| \lVert v  \rVert_{-1/2+\tau,\omega,\Gamma} \ ,
\end{align}
and for $\tau \in (-\frac{1}{2}, \frac{1}{2})$
\begin{align}  
\lVert \mathcal{V}^{\omega} v \rVert_{1/2+\tau,\omega,\Gamma}  =\lVert \gamma_0 K_{0}^{\omega}v \rVert_{1/2+\tau,\omega,\Gamma} \lesssim_\sigma |\omega| \lVert v  \rVert_{-1/2+\tau,\omega,\Gamma} \ ,\label{Vestimate}\\
\lVert \mathcal{K'}^{\omega}v \rVert_{-1/2+\tau,\omega,\Gamma}=\lVert \gamma_{1}K_{0}^{\omega}v \rVert_{-1/2+\tau,\omega,\Gamma} \lesssim_\sigma |\omega| \lVert v  \rVert_{-1/2+\tau,\omega,\Gamma}\ . \label{Kestimate}
\end{align}
From the proof of Lemma \ref{LabelT}, \eqref{K1formula} and the trace theorem we obtain for $\tau \in [-\frac{1}{2}, \frac{1}{2}]$ 
\begin{align}  
\lVert K_{1}^{\omega}v \rVert_{1+\tau,\omega,\Omega} \lesssim_\sigma |\omega|^{2} \lVert v  \rVert_{1/2+\tau,\omega,\Gamma} \ ,
\end{align}
generalizing \eqref{K1estimate}. We conclude for $\tau \in (-\frac{1}{2}, \frac{1}{2})$
\begin{align}  
\lVert \mathcal{K}^{\omega} v \rVert_{1/2+\tau,\omega,\Gamma} = \lVert \gamma_{0}K_{1}^{\omega}v \rVert_{1/2+\tau,\omega,\Gamma} \lesssim_\sigma |\omega|^{2} \lVert v  \rVert_{1/2+\tau,\omega,\Gamma} \ ,\label{Kpestimate}\\
\lVert \mathcal{W}^{\omega}v \rVert_{-1/2+\tau,\omega,\Gamma} = \lVert \gamma_{1}K_{1}^{\omega}v \rVert_{-1/2+\tau,\omega,\Gamma} \lesssim_\sigma |\omega|^{2} \lVert v  \rVert_{1/2+\tau,\omega,\Gamma}\ . \label{Westimate}
\end{align}
Using the Fourier transform to translate back into the time domain, we conclude the proof of Theorem \ref{mapthm2} for $\Gamma$ Lipschitz.

 For $\Gamma$ of class $C^{1,\alpha}$, the Dirichlet and Neumann traces $\gamma_0$ and $\gamma_1$ are also continuous in the endpoints of the interval $\tau \in [-\frac{1}{2}, \frac{1}{2}]$, and the estimates \eqref{Vestimate}, \eqref{Kestimate}, \eqref{Kpestimate} and \eqref{Westimate} extend to the endpoints $\tau = \pm \frac{1}{2}$.  
\end{proof}

\end{document}